\documentclass{amsart}

\usepackage{amsmath}
\usepackage{amsthm}
\usepackage{hyperref}
\usepackage{amsfonts,graphics,amsthm,amsfonts,amscd,latexsym}
\usepackage{epsfig}
\usepackage{flafter}
\usepackage{mathtools}
\usepackage{comment}
\usepackage{stmaryrd}

\hypersetup{
    colorlinks=true,    
    linkcolor=blue,          
    citecolor=blue,      
    filecolor=blue,      
    urlcolor=blue           
}
\usepackage{tikz}
\usetikzlibrary{graphs,positioning,arrows,shapes.misc,decorations.pathmorphing}

\tikzset{
    >=stealth,
    every picture/.style={thick},
    graphs/every graph/.style={empty nodes},
}

\tikzstyle{vertex}=[
    draw,
    circle,
    fill=black,
    inner sep=1pt,
    minimum width=5pt,
]
\usepackage[position=top]{subfig}
\usepackage[alphabetic,backrefs]{amsrefs}
\usepackage{amssymb}
\usepackage{color}

\calclayout

\usetikzlibrary{decorations.pathmorphing}
\tikzstyle{printersafe}=[decoration={snake,amplitude=0pt}]

\newcommand{\rank}{\operatorname{rank}}

\newcommand{\pp}{\mathbb{P}}

\renewcommand{\qq}{\mathbb{Q}}
\newcommand{\zz}{\mathbb{Z}}
\newcommand{\nn}{\mathbb{N}}

\newtheorem{theo}{Theorem}[section]
\newcounter{tmp}

\def\O#1.{\mathcal {O}_{#1}}			
\def\pr #1.{\mathbb P^{#1}}				
\def\af #1.{\mathbb A^{#1}}			
\def\ses#1.#2.#3.{0\to #1\to #2\to #3 \to 0}	
\def\xrar#1.{\xrightarrow{#1}}			
\def\K#1.{K_{#1}}						
\def\bA#1.{\mathbf{A}_{#1}}			
\def\bM#1.{\mathbf{M}_{#1}}				
\def\bL#1.{\mathbf{L}_{#1}}				
\def\bB#1.{\mathbf{B}_{#1}}				
\def\bK#1.{\mathbf{K}_{#1}}			
\def\subs#1.{_{#1}}					
\def\sups#1.{^{#1}}

\usepackage{tikz}
\usetikzlibrary{matrix,arrows,decorations.pathmorphing}

\newtheorem{introdef}{Definition}

  \newtheorem{introthm}{Theorem}

  \newtheorem{theorem}{Theorem}[section]
  \newtheorem{lemma}[theorem]{Lemma}
  \newtheorem{proposition}[theorem]{Proposition}
  \newtheorem{corollary}[theorem]{Corollary}

  \newtheorem{definition}[theorem]{Definition}

\theoremstyle{remark}

\numberwithin{equation}{section}

\usepackage[all]{xy}

\begin{document}

\title[Fano type surfaces with large cyclic automorphisms]{Fano type surfaces with large cyclic automorphisms}

\author[J.~Moraga]{Joaqu\'in Moraga}
\address{Department of Mathematics, Princeton University, Fine Hall, Washington Road, Princeton, NJ 08544-1000, USA
}
\email{jmoraga@princeton.edu}

\maketitle

\begin{abstract}
We give a characterization of Fano type surfaces with large cyclic automorphisms.
\end{abstract}

\setcounter{tocdepth}{1} 
\tableofcontents

\section{Introduction}

Fano type varieties are one of the three building blocks of algebraic varieties.
The research of Fano type surfaces was initiated with the study of smooth Fano surfaces by 
del Pezzo in the late $19$th century.
The work of del Pezzo culminated with the famous classification of 
smooth Fano surfaces, nowadays known as {\em del Pezzo surfaces}.
Such surfaces belong to $9$ families distinguished by their degree.
The Borisov-Alexeev-Borisov conjecture predicted that this boundedness behavior holds
for mildly singular Fano type varieties.
This conjecture was proved in dimension two by Alexeev~\cite{Ale94}.
In~\cite{Bir04}, Birkar proved the theory of complements for surfaces.
He used this theory 
to give a second proof of the boundedness of weak log Fano surfaces~\cite{Bir04}.
In~\cite{Bir16a}, Birkar proved the boundedness of log canonical complements for Fano type varieties
of arbitrary dimension.
This machinery was used to give a positive answer to the BAB conjecture~\cite{Bir16b}. 
A better understanding of the geometry and complements
of Fano type surfaces often guides to important developments of 
higher-dimensional birational geometry.
In general, a novel understanding of Fano type surfaces
is usually the cornerstone for theorems about Fano type varieties
of arbitrary dimension.
We recall that a characterization of projective rational surfaces with a large group action of ${\rm Aut}(X)$ on ${\rm NS}(X)$ is given in~\cite{CD12}.

In this article, we study Fano type surfaces with large cyclic automorphisms.
Fano type surfaces with large cyclic automorphisms
appear naturally when studying Kawamata log terminal $3$-fold singularities
with large Cartier index.
The main aim of this paper is to understand how the existence
of large cyclic automorphisms on Fano type surfaces reflect in their geometry.
Moreover, we study which invariants of the Fano type surface $X$ will control this concept of {\em large}. 
For instance, it is easy to see that del Pezzo surfaces with 
large finite abelian automorphism groups of rank two must be toric, 
i.e., they are either $\pp^1\times\pp^1$ or the blow-up of $\pp^2$ at up to three points.
More generally, we expect the existence of a constant $N(\epsilon)$, only depending on $\epsilon$, satisfying the following.
Given an $\epsilon$-log canonical Fano surface $X$ 
with $\zz_m\leqslant {\rm Aut}(X)$ and $m\geq N(\epsilon)$,
then $X$ can be endowed with a $\mathbb{G}_m$-action.
In the two previous examples, 
we see how the existence of large cyclic automorphisms of $X$ reflects in the geometry of the variety as the existence of a torus action.
In these examples, the invariant of the Fano type surface $X$
that controls the concept of {\em large} for the automorphism group
is the minimal log discrepancy.
At the same time, it is known that if a Fano type surface has a
large finite abelian automorphism group of rank $k$, 
then $k$ is at most two.
This is a consequence of the Jordan property for the birational automorphism group of Fano type varieties (see, e.g.,~\cites{PS14,PS16,PS17}).
Hence, the above examples show the maximal rank behavior.

The main result of this article is a characterization of Fano type surfaces
with large cyclic automorphisms.
Surprisingly, the concept of {\em large} is effective 
and only depends on the dimension.
Thus, it is not necessary to bound any invariant of the Fano type surfaces as in the above examples.
Our first theorem gives a characterization of Fano type surfaces
with large cyclic automorphism groups.

\begin{introthm}\label{introthm-1}
There exists a positive integer $N$ satisfying the following.
Let $X$ be a Fano type surface so that $G:=\zz_m \leqslant {\rm Aut}(X)$ with $m\geq N$.
Then, there exists:
\begin{enumerate}
    \item A subgroup $A\leqslant G$ of index at most $N$,
    \item a boundary $B$ on $X$, and 
    \item a $A$-equivariant birational morphism $X\dashrightarrow X'$,
\end{enumerate}
satisfying the following conditions:
\begin{enumerate}
    \item The pair $(X,B)$ is log canonical, $G$-invariant, and $N(K_X+B)\sim 0$, 
    \item the push-forward of $K_X+B$ to $X'$ is a log pair $K_{X'}+B'$, 
    \item the pair $(X',B')$ admits a $\mathbb{G}_m$-action, and 
    \item there are groups monomorphisms
    $A<\mathbb{G}_m\leqslant {\rm  Aut}(X',B')$.
\end{enumerate}
\end{introthm}

For the convenience of the reader, we introduce some definitions to restate the first theorem less technically.

\begin{introdef}\label{def:log-crepant-torus-action}{\rm 
We say that a log canonical projective pair $(X,B)$ admits a {\em crepant torus action}
if it admits a crepant model $(X',B')$ endowed with a $\mathbb{G}_m$-action.
We require that $(X',B')$ is a log pair.
In particular, the birational morphism $X'\dashrightarrow X$ (resp. $X\dashrightarrow X'$)
only extract divisors with log discrepancies in the interval $[0,1]$ with respect to $(X,B)$ (resp. $(X',B')$).
We say that a projective variety $X$ admits a {\em log crepant torus action}
if $(X,B)$ admits a crepant torus action for some boundary $B$ on $X$.}
\end{introdef}

The following is a more natural way to state our first theorem.

\begingroup
\setcounter{tmp}{\value{theo}}
\setcounter{theo}{0}
\renewcommand\thetheo{\Alph{theo}}
\begin{theo}\label{introthm:A}
A Fano type surface $X$ with a large cyclic automorphism 
admits a log crepant torus action.
\end{theo}

Here, {\em large} means that $|G|$ is larger than a universal constant (as in the statement of Theorem~\ref{introthm-1}).

Now, we turn to introduce a concept to measure the cardinality of generators of finite subgroups of bounded index of a given finite group.
Let $G$ be a finite group.
We define the {\em rank up to index $N$}, denoted by $r_N(G)$, to be the minimum rank among subgroups of index at most $N$.
We define the {\em $k$-generation} order to be the maximum $N$ such that $|G|\geq N$ and $r_N(G)\geq k$.
The $k$-generation of a finite group $G$ is denoted by $g_k(G)$.
Our second theorem gives a characterization of Fano type surfaces
with a finite group of automorphisms with large $2$-generation.
In particular, we obtain a characterization of Fano type surfaces with large abelian automorphism groups of rank two.

\begin{introthm}\label{introthm-2}
There exists a positive integer $N$ satisfying the following.
Let $X$ be a Fano type surface so that $G\leqslant {\rm Aut}(X)$ is a finite group with
$g_2(G)\geq N$.
Then, there exists:
\begin{enumerate}
\item An abelian normal subgroup $A\leqslant G$ of index at most $N$, 
\item a boundary $B$ on $X$, and 
\item a $A$-equivariant birational morphism $X\dashrightarrow X'$, 
\end{enumerate}
satisfying the following conditions:
\begin{enumerate}
    \item The pair $(X,B)$ is log canonical, $G$-invariant, and $K_X+B\sim 0$, 
    \item the push-forward of $K_X+B$ to $X'$ is a log pair $K_{X'}+B'$, 
    \item the pair $(X',B')$ is a log Calabi--Yau toric surface, and 
    \item there are group monomorphisms $A<\mathbb{G}_m^2\leqslant {\rm Aut}(X',B')$.
\end{enumerate}
In particular, $B'$ is the reduced toric boundary.
\end{introthm}

Analogously to the first theorem, we introduce some definitions to state our second theorem less technically.

\begin{introdef}{\rm 
We say that a log canonical projective pair $(X,B)$ is {\em crepant toric}
if $(X,B)$ is crepant to a projective toric pair.
We say that a projective variety $X$ is {\em log crepant toric }
if $(X,B)$ is crepant toric for some boundary $B$ on $X$.}
\end{introdef}

The following is a more natural way to state our second theorem.

\begin{theo}\label{introthm:B}
A Fano type surface with a large finite abelian group of automorphisms of rank two is crepant log toric.
\end{theo}
\endgroup

Here, {\em large} means that $g_2(G)$ is larger than a universal constant (as in the statement of Theorem~\ref{introthm-2}).
The main ingredients for the proofs of the above theorems
are the Jordan property for finite birational automorphism groups of Fano type varieties, the theory of $G$-invariant complements, 
the characterization of toric varieties using complexity, 
the $G$-invariant minimal model program for surfaces,
and a characterization of locally toric surface morphisms due to Shokurov~\cite{Sho00}*{Theorem 6.4}.
The above theorems can be stated in a more general setting
for log pairs $(X,\Delta)$ on a Fano type surface $X$
so that $(X,\Delta)$ is $\qq$-complemented
(see Theorem~\ref{thm:cyclic-automorphism} and Theorem~\ref{thm:rank-2-automorphism}).
Now, we turn to give two applications of the above theorems.

Our first application is related to the fundamental group of the smooth locus of a Fano type surface.
It is known that such group is finite~\cites{FKL93,GZ94,GZ95,KM99}. Moreover, it can encode much of the geometry of the surface~\cites{Xu09}. We prove that if such group is large enough, then $X$ is log crepant to a finite quotient of a toric variety.
We make this concept precise in the following definition.

\begin{introdef}{\em 
We say that a log canonical projective pair $(X,B)$ is a {\em crepant toric quotient} if $(X,B)$ is crepant to a projective pair which is the quotient of a log Calabi--Yau toric pair by a finite automorphism group of the pair.
We say that $X$ is a {\em log crepant toric quotient} if $(X,B)$ is a crepant toric quotient for some boundary $B$ on $X$.
}
\end{introdef}

\begin{introthm}\label{introthm:fundamental}
There exists a positive integer $N$ satisfying the following.
Let $X$ be a klt surface with $-K_X$ ample.
If $G\leqslant \pi_1(X^{\rm sm})$ satisfies $g_2(G)\geq N$, then $X$ is a log crepant toric quotient.
\end{introthm}

In Section~\ref{sec:large-fundamental}, we give a more general version of Theorem~\ref{introthm:fundamental}
which also deals with the case of log pairs.
We will also consider the case of fundamental groups of rank one. 
Furthermore, we prove that there is a bounded cover of $X$ which makes it a log crepant toric variety.

\begin{introdef}{\em
We say that a singularity $x\in(X,B)$ is a {\em toric quotient singularity} if it is the quotient of a singularity on a toric pair by a finite automorphism group.
Note that toric quotient singularities are a generalization of quotient singularities. 
They may not be $\qq$-factorial,
as toric singularities themselves may not be $\qq$-factorial.
Although the small $\qq$-factorialization of a toric quotient singularity has only quotient singularities.
Furthermore, toric quotient singularities are always log canonical.
We say that a singularity $x\in X$ is a {\em log toric quotient singularity} if $x\in (X,B)$ is a toric quotient singularity for some boundary $B$ on $X$.

We say that a log canonical singularity $x\in(X,B)$ is a {\em crepant toric quotient singularity}
if there exists a toric quotient singularity $y\in (Y,B_Y)$ and two projective morphisms 
$\phi_x \colon Z \rightarrow X$ and
$\phi_y \colon Z \rightarrow Y$
so that the following equality holds:
\[
\phi_x^*(K_X+B)=\phi_y^*(K_Y+B_Y).
\]
Finally, we say that a singularity $x\in X$ is a {\em log crepant toric quotient singularity} if $x\in (X,B)$ is a crepant toric quotient singularity for some boundary $B$.
We have natural inclusions between classes of singularities: 
\[
\{\text{Quotient singularities}\} 
\cup 
\{\text{Toric singularities}\} 
\subset
\{\text{Toric quotient singularities}\}
\subset 
\]
\[
\{
\text{Crepant toric quotient singularities}
\}\subset 
\{
\text{Log crepant toric quotient singularities}
\}.
\]
All the above inclusions are strict.
We may use the abbreviations 
{\em tq} (resp. {\em ctq} and {\em lctq})
to denote 
toric quotient (resp. crepant toric quotient and log crepant toric quotient).
}
\end{introdef}

In terms of the minimal model program 
all the above classes of singularities behave similarly.
The following theorem shows that 
lctq singularities and their deformations characterize klt $3$-fold singularities with a large Class group.

\begin{introthm}\label{introthm:large-class}
There exists a positive integer $N$ satisfying the following.
Let $x\in X$ be a $\qq$-factorial klt $3$-fold singularity.
If $G\leqslant {\rm Cl}(X;x)$ satisfies 
$g_3(G)\geq N$, then $x\in X$ degenerates to a lctq singularity.
\end{introthm}

In Section~\ref{sec:larg-class-group}, we give more general versions of Theorem~\ref{introthm:large-class}.
In these versions, we also consider the case of rank two.
The case of rank one is out of the scope of this article.

It is expected that all the theorems proved in this article can be generalized to higher dimensions.
The author will deal with the higher-dimensional cases in future papers.
This article hopes to be a foundation for the study
of Fano type varieties with large abelian automorphism groups and the corresponding applications to the study of klt singularities.

\subsection*{Acknowledgements}
The author would like to thank Caucher Birkar, Christopher Hacon, J\'anos Koll\'ar and Constantin Shramov for many useful comments.

\section{Preliminaries}

In this section, we recall the singularities of the minimal model program, introduce the theory of $G$-invariant complements, and
prove some preliminary results.
Throughout this article, we work over an algebraically closed field $\mathbb{K}$ of characteristic zero.
$\mathbb{G}_m^k$ stands for the $k$-dimensional $\mathbb{K}$-torus.
We will use some classic results of toric geometry over $\mathbb{K}$.
See~\cites{Ful93,CLS11,Cox95} for some references on toric geometry.

\subsection{Singularities of the minimal model program}

In this subsection, we recall classic definitions regarding the singularities of the minimal model program.
We follow the classic definitions for singularities of the minimal model program~\cites{KM98,Kol13,HK13}.

\begin{definition}
{\em A {\em contraction} is a morphism of quasi-projective varieties
$\phi\colon X\rightarrow Y$ so that $\phi_*\mathcal{O}_X=\mathcal{O}_Y$.
Note that $Y$ is normal provided that $X$ is normal.
A {\em fibration} is a contraction with positive dimensional general fiber.}
\end{definition}

\begin{definition}
{\em 
Let $\mathcal{Q}$ be a subset of $[0,1]$.
We define the set of {\em hyperstandard coefficients} associated to $\mathcal{Q}$,
denoted by $\mathcal{H}(\mathcal{Q})$, to be
\[
\left\{ 1-\frac{q}{m} \mid q\in \mathcal{Q}, m\in \mathbb{N}\right\}.
\]
If $\mathcal{Q}=\{0,1\}$, 
we say that $\mathcal{H}(\mathcal{Q})$ is the set of {\em standard coefficients}.
We say that a log pair $(X,B)$ has {\em standard coefficients} if the coefficients of $B$ are standard.
}
\end{definition}

\begin{definition}{\em 
A {\em log pair} $(X,\Delta)$ consist of a normal quasi-projective variety $X$ and an effective $\qq$-divisor $\Delta$,
such that the $\qq$-divisor $K_X+\Delta$ is $\qq$-Cartier.
Given a projective birational morphism $\pi\colon Y\rightarrow X$ from a normal quasi-projectiuve variety $Y$ and
a prime divisor $E$ on $Y$, 
we define the {\em log discrepancy} of the log pair $(X,\Delta)$ at the prime divisor $E$ to be
\[
a_E(X,\Delta):=1-{\rm coeff}_E(\pi^*(K_X+\Delta)).
\]
We say that a log pair $(X,\Delta)$ is {\em Kawamata log terminal} (resp. {\em log canonical}) if every log discrepancy is positive (resp. non-negative).
We may use the usual abbreviation lc (resp. klt) for log canonical (resp. Kawamata log terminal).
}
\end{definition}

\begin{definition}{\em 
Let $(X,\Delta)$ be a pair with log canonical singularities.
A prime divisor $E$ over $X$ is said to be a {\em log canonical place} if the corresponding log discrepancy equals zero.
The image of a log canonical place 
in $X$ is said to be a {\em log canonical center}.
}
\end{definition}

\begin{definition}
{\em
A log pair $(X,\Delta)$ is said to be {\em divisorially log terminal} or {\em dlt} if there is an open set $U\subset X$ which satisfies the following properties:
\begin{enumerate}
    \item the coefficients of $\Delta$ are at most one,
    \item $U$ is smooth and $\Delta|_U$ has simple normal crossing, and
    \item any log canonical center of $(X,\Delta)$ intersect $U$ non-trivially and is given by intersection of the strata of $\lfloor \Delta \rfloor$.
\end{enumerate}
Given a log canonical pair $(X,\Delta)$, we say that $\pi \colon Y\rightarrow X$ is a {\em $\qq$-factorial dlt modification}
if $Y$ is $\qq$-factorial, $\pi$ only extract log canonical places, and $\pi^*(K_X+\Delta)=K_Y+\Delta_Y$ is a divisorially log terminal pair.
The existence of $\qq$-factorial dlt modifications for log canonical pairs is well-known~\cite{KK13}.
We say that a dlt pair $(X,\Delta)$ is {\em purely log terminal} (or {\em plt} for short) if $(X,\Delta)$ has at most one log canonical center.
}
\end{definition}

\begin{definition}
{\em 
Let $(X,\Delta)$ be a log canonical pair and $D$ be a $\qq$-Cartier effective divisor on $X$.
We define the {\em log canonical threshold} of $D$ on $(X,\Delta)$, 
denoted by ${\rm lct}((X,\Delta);D)$,
to be the maximum positive rational number $t$, such that $K_X+\Delta+tD$ is log canonical.
}
\end{definition}

\begin{definition}{\em 
Let $(X,\Delta)$ be an lc pair and $x\in X$ a closed point. A {\em purely log terminal blow-up} of  $(X,\Delta)$ at $x\in X$ is a birational morphism $\pi \colon Y\rightarrow X$, such that
\begin{enumerate}
\item $\pi$ has a unique exceptional divisor $E$
\item the center of $E$ on $X$ is $x$,
\item $-E$ is ample over $X$, and
\item $(Y,\Delta_Y+E)$ is purely log terminal near $E$.
\end{enumerate}
Here, $\Delta_Y$ is the strict transform of $\Delta$ on $Y$.
We may write {\em plt blow-up} instead of {\em purely log terminal blow-up} to shorten the notation.
}
\end{definition}

\begin{definition}{\em 
Let $M$ be a positive integer.
Let $(X,\Delta)$ be a log canonical pair.
We say that an effective divisor $B\geq \Delta$ is a {\em $M$-complement} of the pair $(X,\Delta)$ if the following conditions hold:
\begin{enumerate}
    \item $(X,B)$ is a log canonical pair, and
    \item $M(K_X+B)\sim 0$.
\end{enumerate}
Note that if $(X,\Delta)$ admits a $M$-complement, then 
$-(K_X+\Delta)$ is a pseudo-effective divisor.
In the above case, we say that $(X,\Delta)$ is $M$-complemented.
We may also say that $-(K_X+\Delta)$ is $M$-complemented.

We say that an effective divisor $B\geq \Delta$ is a
{\em $\qq$-complement} of the pair $(X,\Delta)$,
if the following conditions hold:
\begin{enumerate}
    \item $(X,B)$ is a log canonical pair, and
    \item $K_X+B\sim_\qq 0$.
\end{enumerate}
Note that any $M$-complement is a $\qq$-complement.
In the above case, we say that $(X,\Delta)$ is $\qq$-complemented.
We may also say that $-(K_X+\Delta)$ is $\qq$-complemented.
}
\end{definition}

\begin{definition}{\em 
Let $X$ be a quasi-projective variety and $X\rightarrow Z$ be a projective contraction.
We say that $X$ is of {\em Fano type} over $Z$ if there exists a  boundary $\Delta$  big over $Z$ so that 
$(X,\Delta)$ is klt and
$K_X+\Delta\sim_{\qq,Z} 0$.
Recall that relatively Fano type varieties are relative Mori dream spaces over the base~\cite{BCHM10}.
In particular, every minimal model program for a divisor $D$ on $X$ over $Z$ terminates
with a good minimal model or a Mori fiber space over $Z$.

Let $X$ be a quasi-projective variety and $X\rightarrow Z$ be a projective contraction.
We say that $X$ is of {\em log Calabi--Yau type} over $Z$ if there exists a boundary $\Delta$ on $X$ so that 
$(X,\Delta)$ is log canonical and
$K_X+\Delta\sim_{\qq,Z} 0$.
If $Z={\rm Spec}(\mathbb{K})$, then we just say that $(X,B)$ is a log Calabi--Yau pair.
}
\end{definition}

\begin{definition}
{\em 
Let $X$ be a projective variety, $D$ be a prime divisor on $X$, and $\phi\colon X \rightarrow Z$ be a fibration.
We say that $D$ is {\em $\phi$-horizontal} if $D$ dominates $Z$.
Respectively, we say that $D$ is {\em $\phi$-vertical} if $D$ does not dominate $Z$.
We may say {\em vertical over $Z$} (resp.
{\em horizontal over $Z$}) instead of $\phi$-vertical (resp. $\phi$-horizontal)
when the morphism is not labeled.
When $\phi$ and $Z$ are clear from the context, we may just say that $D$ is vertical or horizontal. 
}
\end{definition}

\begin{definition}
{\em 
Let $(X,\Delta)$ be a log pair.
We say that $(X,\Delta)$ {\em admits a $\mathbb{G}_m$-action} if $X$ admits a $\mathbb{G}_m$-action so that
$\Delta$ is invariant.
}
\end{definition}

\begin{definition}
{\em
Let $(X,\Delta)$ be a log canonical pair.
We say that $(Y,\Delta_Y)$ is a {\em crepant model} (or {\em crepant transformation}) of $(X,\Delta)$ if
there is a birational map $\pi\colon Y\rightarrow X$,
which only extract divisors of with log discrepancy at most one with respect to the log pair $(X,\Delta)$,
and $\pi^*(K_X+\Delta)=K_Y+\Delta_Y$.
}
\end{definition}

\begin{definition}
{\em 
Let $X$ be a demi-normal scheme.
Denote by $D$ the conductor of $X$.
Let $B$ be an effective divisor on $X$ whose support does not contain a prime component of the conductor.
Let $\pi \colon X^\eta \rightarrow X$ be the normalization morphism of $X$.
Let $B^\eta$ be the divisorial part of $\pi^{-1}(B)$.
We say that $(X,B)$ is {\em semi log canonical} if $K_X+B$ is $\qq$-Cartier and $K_{X^\eta}+B^\eta+D^\eta$ is log canonical.
We write {\em slc} to abbreviate semi-log canonical.
}
\end{definition}

\subsection{Theory of $G$-invariant complements}

In this subsection, we introduce $G$-invariant complements and prove the existence of bounded $G$-invariant complements.
In what follows, we may use the $G$-equivariant version of some statements of the minimal model program.
See~\cites{CS11a,CS11b,CS12} for some results related to $G$-equivariant singularities of the minimal model program.

\begin{definition}{\em 
Let $X$ be a projective algebraic variety and
$G\leqslant {\rm Aut}(X)$ be a finite automorphism group.
We say that a log pair $(X,\Delta)$ is {\em $G$-invariant} if $g^*(\Delta)=\Delta$ for every element $g\in G$.
Note that this is equivalent to 
${\rm coeff}_P(\Delta)={\rm coeff}_{g^*(P)}(\Delta)$ for every prime divisor $P$ on $X$ and for every element $g\in G$.
Hence, we can define the coefficients of $\Delta$ over the prime $G$-orbits of $X$.
}
\end{definition}

\begin{definition}{\em 
Let $M$ be a positive integer.
Let $X$ be a projective algebraic variety and $G\leqslant {\rm Aut}(X)$ be a finite automorphism group.
Let $(X,\Delta)$ be a $G$-invariant log canonical pair.
We say that an effective divisor $B\geq \Delta$ is a {\em $G$-invariant $M$-complement} of the pair $(X,\Delta)$, if the following conditions hold:
\begin{enumerate}
    \item $(X,B)$ is $G$-invariant,
    \item $(X,B)$ is a log canonical pair, and
    \item $M(K_X+B)\sim 0$.
\end{enumerate}
Note that if $(X,\Delta)$ admits a $G$-invariant $M$-complement, then 
$-(K_X+\Delta)$ is a pseudo-effective divisor.

We say that an effective divisor $B\geq \Delta$ is a
{\em $G$-invariant $\qq$-complement} of the pair $(X,\Delta)$,
if the following conditions hold:
\begin{enumerate}
    \item $(X,B)$ is $G$-invariant,
    \item $(X,B)$ is a log canonical pair, and
    \item $K_X+B\sim_\qq 0$.
\end{enumerate}
Note that any $G$-invariant $M$-complement is a $G$-invariant $\qq$-complement.
We may write $G\qq$-complement instead of $G$-invariant $\qq$-complement to abbreviate.
}
\end{definition}

\begin{proposition}\label{prop:quotient-FT}
The quotient of a Fano type surface by a finite automorphism group is of Fano type.
\end{proposition}

\begin{proof}
Let $\phi\colon X \rightarrow Y$ be the quotient of a Fano type surface by a finite automorphism group.
Let $\Delta$ be a boundary on $X$ so that
$K_X+\Delta \sim_\qq 0$ and the pair
$(X,\Delta)$ has klt singularities.
Let $B:=\frac{1}{|G|}\sum_{g\in G}g^*\Delta$.
The boundary $B$ is $G$-invariant,
$K_X+B\sim_\qq 0$, and $(X,B)$ has klt singularities.
Let $(Y,B_Y)$ be the quotient of the log pair
$(X,B)$ by $G$.
Then, we have that $(Y,B_Y)$ has klt singularities, $B_Y$ is a big divisor, 
and $K_Y+B_Y\sim_\qq 0$.
We conclude that $Y$ is a Fano type surface.
\end{proof}

\begin{proposition}\label{prop:q-compl}
Let $X$ be a Fano type variety and
$G\leqslant {\rm Aut}(X)$ be a finite subgroup.
Let $(X,\Delta)$ be a log canonical $G$-invariant pair.
Assume that $(X,\Delta)$ admits a $\qq$-complement.
Then $(X,\Delta)$ admits a $G$-invariant $\qq$-complement.
\end{proposition}

\begin{proof}
Let $B\geq \Delta$ be a $\qq$-complement for $(X,\Delta)$.
Consider the divisor $\Gamma:=\frac{1}{|G|}\sum_{g\in G}g^*B$.
Note that $\Gamma \geq \Delta$.
Indeed, for every element $g\in G$ we have
$g^*B\geq g^*\Delta=\Delta$, so 
$\sum_{g\in G}g^*B\geq |G|\Delta$, hence
$\Gamma\geq \Delta$.
The pair $(X,\Gamma)$ is log canonical and $G$-invariant.
Furthermore, by construction, we have that $K_X+\Gamma\sim_\qq 0$.
We conclude that $\Gamma\geq \Delta$ is a $G$-invariant $\qq$-complement of $(X,\Delta)$.
\end{proof}

\begin{proposition}\label{prop:hurwitz}
Let $X$ be an algebraic variety and $G\leqslant {\rm Aut}(X)$ a finite automorphism group.
Let $(X,\Delta)$ be a $G$-invariant log canonical pair.
Let $\pi\colon X\rightarrow Y$ be the quotient by $G$.
Then, there exists a log canonical pair $(Y,\Delta_Y)$ so that $\pi^*(K_Y+\Delta_Y)=K_X+\Delta$.
Furthermore, the coefficients of $\Delta_Y$ have the form 
$\frac{1-m+\lambda}{m}$, where $m$ is a positive integer and $\lambda$ is a coefficient of $\Delta$.
\end{proposition}

\begin{proof}
First, we prove the statement for $K_X$.
By the Hurwitz formula, we can write
$K_X=\pi^*(K_Y)+R$ where $R$ is the ramification divisor.
Let $P$ be a prime divisor of $X$ contained in the support of $R$.
Then, for every element $g\in G$, we have that
${\rm coeff}_{g^*P}(X)={\rm mult}_P(\pi)=
{\rm ram}_P(\pi)+1$.
This means that all the prime divisors on the $G$-orbit of $P$ appear with the same coefficient on $R$.
Let $P_1,\dots,P_k$ be the image of all the $G$-orbits of prime components of $R$.
Let $m_i={\rm mult}_{P_i}(\pi)$ be the multiplicity of $\pi$ at such prime $G$-orbits.
Then, we have that $\pi^*\left( \sum_{i=1}^k \frac{m_i-1}{m_i}P_i\right)=R$.
We conclude that
\[
K_X=\pi^*\left(K_Y+\sum_{i=1}^k \frac{m_i-1}{m_i}P_i \right).
\]
The above pull-back formula also follows from~\cite{Sho92}*{Proposition 2.1}.
Now, let's consider the $G$-invariant boundary $\Delta$.
For every $G$-orbit of a prime divisor $P$ on $X$, we define $b_i:={\rm coeff}_{P}(\Delta)$ 
and $n_i:={\rm mult}_P(\pi)$.
Let $Q_1,\dots,Q_s$ be the images of all the $G$-orbits contained in the support of $\Delta$.
Then, we have that
$\Delta=\pi^*( \sum_{i=1}^s \frac{b_i}{n_i}Q_i)$.
Thus, we have an equality 
\begin{equation}\label{eq:pull-back-finite}
K_X+B = 
\pi^*\left(
K_Y+\sum_{i=1}^k \frac{m_i-1}{m_i}P_i 
+\sum_{i=1}^s \frac{b_i}{n_i}Q_i
\right)
\end{equation}
We define
\[
\Delta_Y:= \frac{m_i-1}{m_i}P_i 
+\sum_{i=1}^s \frac{b_i}{n_i}Q_i.
\]
The fact that $(Y,\Delta_Y)$ is a log canonical pair follows from~\cite{Sho92}*{Proposition 2.2}.
Note that equation~\eqref{eq:pull-back-finite} implies that the coefficients of $\Delta_Y$ have the form $\frac{m-1+\lambda}{m}$ where $m$ is a positive integer and $\lambda$ is a coefficient of $\Delta$. 
Indeed, if $n_i>1$ for some $i\in \{1,\dots,s\}$, then $P_j=Q_i$ for some $j$
and $m_i=n_j$.
\end{proof}

The following theorem is the boundedness of $G$-invariant complements for Fano type varieties.
It is a consequence of the boundedness of complements proved in~\cite{Bir16a}.

\begin{theorem}\label{thm:g-inv-compl}
Let $n$ be a positive integer.
Let $\Lambda\subset \qq$ be a set satisfying
the descending chain condition with rational accumulation points.
There exists a positive integer $M:=M(\Lambda,n)$, only depending on $\Lambda$ and $n$, satisfying the following.
Let $X$ be a $n$-dimensional Fano type variety and
$\Delta$ a boundary on $X$ such that the following conditions hold:
\begin{enumerate}
    \item $G\leqslant {\rm Aut}(X)$ is a finite subgroup,
    \item $(X,\Delta)$ is log canonical and $G$-invariant,
    \item the coefficients of $\Delta$ belong to $\Lambda$, and
    \item $-(K_X+\Delta)$ is $\qq$-complemented.
\end{enumerate}
Then, there exists a boundary $B\geq \Delta$ on $X$, such that the following holds:
\begin{enumerate}
\item $(X,B)$ is $G$-invariant,
\item $(X,B)$ is log canonical, and
\item $M(K_X+B)\sim 0$.
\end{enumerate}
\end{theorem}

\begin{proof}
Let $(Y,\Delta_Y)$ be the quotient of $(X,\Delta)$ by $G$.
Here, $(Y,\Delta_Y)$ is the pair constructed in Proposition~\ref{prop:hurwitz}.
In particular, we have $K_X+\Delta=\pi^*(K_Y+\Delta_Y)$.
By Proposition~\ref{prop:quotient-FT}, we know that $Y$ is a Fano type variety.
By Proposition~\ref{prop:q-compl}, we know that $(X,\Delta)$ admits a $G\qq$-complement.
Hence, $(Y,\Delta_Y)$ admits a $\qq$-complement $\Gamma_Y\geq \Delta_Y$.
In particular, it is a Mori dream space.
Note that $-(K_Y+\Delta_Y)$ is a pseudo-effective divisor and the log pair $(Y,\Delta_Y)$ is log canonical.
We claim that the coefficients of $\Delta_Y$ belongs to a set $\mathcal{H}(\Lambda) \subset \qq$ which only depends on $\Lambda$, satisfies the descending chain condition, 
and has rational accumulation points.
By Proposition~\ref{prop:hurwitz}, the coefficient of $\Delta_Y$ at a prime divisor $P$ has the form
\[
{\rm coeff}_P(\Delta_Y):=
1-\frac{1}{m}+\frac{\lambda}{m}.
\]
Here, $\lambda\in \Lambda$
and $m$ is the multiplicity index
of $X\rightarrow Y$ at prime component of $X$.
Note that it suffices to find a $M$-complement for $(Y,\Delta_Y)$
and pull it back to $X$.
Moreover, we may run a minimal model program for $-(K_Y+\Delta_Y)$.
This minimal model program is denoted by $Y\dashrightarrow Z$.
It terminates given that $Y$ is a Mori dream space.
Since the divisor $-(K_Y+\Delta_Y)$ is pseudo-effective, then the divisor $-(K_Z+\Delta_Z)$ is semiample.
Note that $(Z,\Delta_Z)$ remains log canonical
and $\qq$-complemented.
Indeed, all the steps of this minimal model program are $\qq$-trivial for the log canonical pair $(Y,\Gamma_Y)$ and $\Gamma_Z\geq\Delta_Z$.
By~\cite{Bir16a}*{Proposition 6.1.(3)}, it suffices to produce a $M$-complement for
$(Z,\Delta_Z)$.
The existence of a $M$-complement for $(Z,\Delta_Z)$ is proved in~\cite{FM18}*{Theorem 1.2}.
Furthermore, $M$ only depends on $\mathcal{H}(\Lambda)$ and $n$, hence it only depends on $\Lambda$ and $n$.
\end{proof}

\begin{proposition}\label{prop:G-equiv-cbf}
Let $M$ be a positive integer.
There exists a positive $M':=M'(M)$, only depending on $M$, satisfying the following.
Let $X$ be a projective surface
and $G\leqslant {\rm Aut}(X)$ a finite automorphism group.
Let $(X,B)$ be a $G$-invariant log canonical pair so that
$M(K_X+B)\sim 0$.
Let $\phi \colon X\rightarrow C$ be a $G$-equivariant fibration to a curve $C$.
Then, we can write
\[
K_X+B \sim \phi^*(M'(K_C+B_C)),
\]
where $M'(K_C+B_C)\sim 0$ and 
$(C,B_C)$ is invariant with respect to the action induced on the base.
\end{proposition}

\begin{proof}
We have an exact sequence
\[
1\rightarrow G_f \rightarrow G \rightarrow G_b \rightarrow 1, 
\]
where $G_b$ is the quotient group acting on $C$ and $G_f$ is the subgroup of $G$ acting fiber-wise.
We obtain a commutative diagram as follows:
\[
 \xymatrix{
X \ar[r]^-{\pi} \ar[d]_-{\phi} & Y \ar[d]^-{\phi_Y} \\
C \ar[r]^-{\pi_C} & C_Y \\ 
 }
\]
Here, $Y$ is the quotient of $X$ by $G$
and $C_Y$ is the quotient of $C$ by $G_b$.
Observe that the coefficients of the log pair
$(Y,B_Y)$ belong to the hyperstandard set $\mathcal{H}\left(\nn\left[\frac{1}{M} \right]\cap [0,1]\right)$.
Note that this set only depends on $M$.
Hence, we may apply the effective canonical bundle formula (see, e.g.,~\cite{PS01}). We obtain a log pair $(C_Y,B_{C_Y})$ with 
\[
K_Y+B_Y \sim \phi^*_Y(M'(K_{C_Y}+B_{C_Y}))
\]
where $M'(K_{C_Y}+B_{C_Y})\sim 0$.
Here, $M'$ only depends on $\mathcal{H}\left(\nn\left[\frac{1}{M} \right]\cap [0,1]\right)$, which only depends on $M$.
Hence, the pair obtained by pulling-back 
$K_{C_Y}+B_{C_Y}$ to $C$ satisfies the conditions of the statement.
\end{proof}

\subsection{$k$-generation order}

In this subsection, we collect some lemmas about the $k$-generation order
and recall the Jordan property for finite birational automorphism groups of Fano type varieties.

\begin{definition}{\em 
We define the {\em rank up to index $N$},
denoted by $r_N(G)$, to be the minimum rank among subgroups of $G$ of index at most $N$.
If $r_N(G)=k$, then any subgroup generated by at most $k-1$ elements has index larger than $N$.
In particular, if $r_N(G)\geq k$, then the minimum rank of $G$ is at least $k$.
Note that 
$r_1(G)=\rank(G)$ 
and $r_N(G)=1$ for every $N\geq |G|$.
In general, $r_N(G)\in \{1,\dots,\rank(G)\}$ for every positive integer $N$.
}
\end{definition}

\begin{definition}{\em 
Let $G$ be a finite group.
We define the {\em $k$-generation order},
denoted by $g_k(G)$, 
to be:
\[
\max \{ N \mid |G|\geq N \text{ and } r_N(G)\geq k\}.
\]
Note that any finite group satisfies $g_1(G)=|G|$.
So, having large one-generation is the same as having large cardinality.
}
\end{definition}

\begin{lemma}\label{lem:k-gen-subgroup}
Let $G$ be a finite group with $g_k(G)\geq N$.
Let $A\leqslant G$ be a subgroup of index at most $J$.
Then $g_k(A)\geq N/J$.
\end{lemma}

\begin{proof}
Note that $|A|\geq N/J$.
Assume that $r_{N/J}(A)<k$.
Then, there exists a subgroup $H\leqslant A$ of index $\leq N/J$ which is generated by strictly less than $k$ elements.
Note that $H\leqslant G$ is a subgroup of index $\leq N$ which is generated by strictly less than $k$ elements.
This contradicts the fact that $r_N(G)\geq k$.
We conclude that $|A|\geq N/J$ and 
$r_{N/J}(A)\geq k$.
Thus, we have that $g_k(A)\geq N/J$.
\end{proof}

\begin{lemma}\label{lem:cyclic-2-gen}
Let $G$ be a finite group with $g_2(G)\geq N$.
Assume that we have an exact sequence
\[
1\rightarrow \zz_n \rightarrow G \rightarrow \zz_m \rightarrow 1.
\]
Then, $\min\{m,n\}\geq N$.
\end{lemma}

\begin{proof}
If $m<N$, then $\zz_n$ is a cyclic subgroup of $G$ of index bounded by $N$, contradicting the fact that $r_N(G)\geq 2$.
On the other hand, if $n<N$, we can consider an element $h$ which maps to $1_m \in \zz_m$.
Let $H\leqslant G$ be the cyclic subgroup generated by $h$.
Note that the order of $h$ in $G$ is at least $m$.
Hence, the index of $H$ in $G$ is at most $n$.
We conclude that $H$ is a cyclic subgroup of $G$ with index bounded by $N$.
This contradicts the fact that $r_N(G)\geq 2$.
\end{proof}

\begin{lemma}\label{lem:g_k-under-quot}
Let $G$ be a finite group with $g_k(G)\geq N$ and $k\geq 2$.
Consider the exact sequence
\[
1\rightarrow C \rightarrow G \rightarrow G'\rightarrow 1,
\]
where $C$ is a cyclic group.
Then, $g_{k-1}(G')\geq N$.
\end{lemma}

\begin{proof}
Assume that $|G'|\leq N$.
Then, $C$ is a cyclic group of index bounded by $N$ on $G$.
This contradicts the fact that $r_N(G)\geq 2$.
Hence, we have that $|G'|>N$.
On the other hand, assume that
$r_N(G')< k-1$.
This means that we can find a subgroup
$H'$ of $G'$ generated by strictly less than $k-1$ elements with index at most $N$.
Let $H_0$ be a subgroup of $G$ surjecting onto $H'$.
We may assume that $H_0$ is generated by strictly less than $k-1$ elements as well.
Let $H$ be the subgroup of $G$ generated by the image of the generator of $C$ and $H_0$.
Then, $H$ is generated by strictly less than $k$ elements.
Furthermore, $H$ has index at most $N$ on $G$,
leading to a contradiction with the fact that $r_N(G)\geq k$.
We conclude that $r_N(G')\geq k-1$.
Hence, we have that $g_{k-1}(G')\geq N$.
\end{proof}

\begin{lemma}\label{lem:g2-supgroup}
Let $G$ be a finite group 
and let $H$ be a subgroup.
If $g_2(H)\geq N$, then $g_2(G)\geq N$.
\end{lemma}

\begin{proof}
Note that $g_2(H)\geq N$ implies that $|H|\geq N$, so $|G|\geq N$.
We need to check that $r_N(G)\geq 2$.
Assume the opposite holds, i.e, 
there is a cyclic subgroup $\zz_m \subset G$ of index at most $N$.
Then $\zz_m \cap H$ is a cyclic subgroup of $H$ of index at most $N$, contradicting the fact that $r_H(G)\geq 2$.
Hence, we have that $r_N(G)\geq 2$.
Thus, $g_2(G)\geq 2$.
\end{proof}

\begin{lemma}\label{lem:subgrp-glnz}
There exists a constant $l:=l(r,k,J)$, only depending on $r,k$ and $J$, satisfying the following.
Let $G<{\rm Aut}(\mathbb{G}_m^k)$ be a finite subgroup.
Let $A_0 \leqslant G$ be an abelian normal subgroup of rank $r$ with index at most $J$.
Then, there exists an abelian characteristic subgroup 
$A\leqslant G$ of index at most $l$ so that $A<\mathbb{G}_m^k$.
\end{lemma}

\begin{proof}
We have an exact sequence
\[
1\rightarrow \mathbb{G}_m^k \rightarrow {\rm Aut}(\mathbb{G}_m^k) \rightarrow 
{\rm Gl}_k(\zz)\rightarrow 1.
\]
There exists a constant $h:=h(k)$, only depending on $k$, such that every finite order element of ${\rm Gl}_k(\zz)$ has order at most $h$.
Set $l:=h!J$.
We have that 
$G^l=\{ g^l \mid g \in G\}$ is a characteristic subgroup of $G$.
By construction, we have that
$G^l < \mathbb{G}_m^k$.
Furthermore, we have that $G^l\leqslant A_0$.
In particular, the group $G^l$ is an abelian characteristic subgroup of $G$.
We need to check that it has bounded index on $G$.
We have subgroups
\[
A_0^l \leqslant G^l \leqslant A_0 \leqslant G.
\]
Note that the index of $A_0^l$ in $A_0$ is at most $l^r$.
We conclude that the index of $A:=G^l$ in $G$ is bounded by $l^rJ$.
\end{proof}

\begin{definition}
{\em 
Let $\mathcal{G}$ be a family of finite groups.
We say that $\mathcal{G}$ is {\em uniformly Jordan} if there is a constant $J(\mathcal{G})$  so that for any group $G\in \mathcal{G}$ there is a normal abelian subgroup $A\leqslant G$ of index at most $J(\mathcal{G})$.
}
\end{definition}

\begin{theorem}{\rm(cf.~\cite{PS16}*{Theorem 1.8})}\label{thm:Jordan}
The family of finite birational automorphism groups of Fano type varieties of dimension $n$ is uniformly Jordan.
\end{theorem}

\subsection{Del Pezzo surfaces with large automorphisms}

In this subsection, we prove some propositions about del Pezzo surfaces with large automorphisms group.

\begin{lemma}\label{lem:ab-action-p1}
Let $M$ be a positive integer.
There exists a positive integer $N:=N(M)$, 
only depending on $M$,
satisfying the following.
Let $A< {\rm Aut}(\pp^1)$ be a finite abelian group.
Let $B$ be a $A$-invariant $M$-complement of $\pp^1$.
If $|A|\geq N$, then up to an automorphism of $\pp^1$, we have that
$B$ is the toric boundary, and
$A\simeq \zz_{|A|}$ acts as the multiplication of a $|A|$-root of unity on the torus.
\end{lemma}

\begin{proof}
If $|A|>6$, then $A$ is a cyclic group acting as the multiplication by a root of unity on a big torus of $\pp^1$.
Since $A$ ramifies with index $|A|$ at zero and infinity, we conclude that $B$ must the sum of those divisors with reduced structure.
\end{proof}

\begin{lemma}\label{lem:ab-action-p2}
Let $M$ be a positive integer.
There exists a positive integer $N:=N(M)$,
only depending on $M$,
satisfying the following.
Let $A< {\rm Aut}(\pp^2)$ be a finite abelian group.
Let $B$ be a $A$-invariant $M$-complement on $\pp^2$.
If $g_2(A)\geq N$, then up to an automorphism of $\pp^2$, we have that $B$ is the toric boundary, 
$A\simeq \zz_{n_1}\oplus \zz_{n_2}$
acts as the multiplication by roots of unity on the torus,
and $\min\{n_1,n_2\}\geq N$.
\end{lemma}

\begin{proof}
Note that $g_2(A)\geq N$ implies that $|A|\geq N$.
Hence, for $N$ large enough, we may assume that $A$ acts as the multiplication by roots of unity on the big torus of $\pp^2$.
In particular, it has rank two.
Let $C\subset \pp^2$ be a reduced curve.
If $C$ is $A$-invariant but is not $\mathbb{G}_m^2$-invariant, then
$C$ has degree at least $N$.
If $C$ is contained in the support of $B$, then $C$ appears with coefficient at least $\frac{1}{M}$.
Thus, we conclude that for $N\geq 4M$, the support of $B$ is $\mathbb{G}_m^2$-invariant.
Hence, it must be the reduced torus invariant boundary.
The fact that $\min\{n_1,n_2\}\geq N$ follows from Lemma~\ref{lem:cyclic-2-gen}.
\end{proof}

\begin{proposition}\label{prop:g-inv-del-pezzo}
Let $M$ be a positive integer.
There exists a positive integer $N:=N(M)$,
only depending on $M$, 
satisfying the following.
Let $X$ be a del Pezzo surface
and $G\leqslant {\rm Aut}(X)$ a finite subgroup with 
$g_2(G)\geq N$.
Let $B$ be an effective divisor on $X$
satisfying the following conditions:
\begin{enumerate}
\item $(X,B)$ is $G$-invariant
\item $(X,B)$ is log canonical, and 
\item $M(K_X+B)\sim $0.
\end{enumerate}
Then $(X,B)$ is a log Calabi--Yau toric pair.
Furthermore, there exists an abelian characteristic subgroup $A\leqslant G$ with
index bounded by $N$, such that
$A< \mathbb{G}^2_m \leqslant {\rm Aut}(X,B)$.
\end{proposition}

\begin{proof}
By the Jordan property, we may find an abelian normal subgroup $A_0 \leqslant G$ of index at most $J(2)$.
Here, $J:=J(2)$ is a constant which only depends on the dimension.
By Lemma~\ref{lem:k-gen-subgroup}, 
we know that $g_2(A_0)\geq g_2(G)/J$.

From the classification of automorphisms of del Pezzo surfaces (see, e.g.,~\cites{Dol12,Hos96, MH74}), 
we know that $X$ must be the blow-up of $\pp^2$ at up to four points
with no three of them collinear.
Otherwise, the automorphism group of $X$ is finite.
These surfaces have a bounded number of $(-1)$-curves and every automorphism will act as a permutation group of such curves.
Let $S$ be the set of such curves and $k:=|S|$.
Hence, we have an exact sequence
\[
1\rightarrow A_1 \rightarrow A_0 \rightarrow S_k.
\]
By Lemma~\ref{lem:k-gen-subgroup}, we have that $g_2(A_1)\geq g_2(G)/Jk!$.
Note that $A_1$ fixes all the curves of the set $S$.
Hence, $A_1$ acts as an automorphism group of $(\pp^2,B_{\pp^2})$, where $B_{\pp^2}$ is the push-forward of $B$ to $\pp^2$.
Furthermore, $A_1$ leaves the images of the $(-1)$-curves which are contracted to $\pp^2$ fixed.
By Lemma~\ref{lem:ab-action-p2},
up to an automorphism of $\pp^2$, we may assume that $A_1$ is acting as the multiplication by roots of unity on the big torus $\mathbb{G}_m^2\subset \pp^2$.
In particular, $A_1$ has rank two.
By Lemma~\ref{lem:cyclic-2-gen}, we conclude that $A_1\simeq \zz_{n_1}\oplus\zz_{n_2}$, where $\min\{n_1,n_2\}\geq g_2(G)/Jk!$.
For $g_2(G)$ large enough, this implies that $(\pp^2,B_{\pp^2})$ is a log Calabi--Yau toric pair.
We know that the exceptional curves of $X\rightarrow \pp^2$ must map to $A_1$-invariant points.
Thus, they must map to the torus invariant points.
Hence, $(X,B)$ is a log Calabi--Yau toric pair as well.
We conclude that the pair $(X,B)$ is a log Calabu--Yau toric pair provided that
$g_2(G)\geq 4MJk!$.
Finally, observe that $G$ fixes the toric boundary $B$, 
hence it induces an automorphism of the torus $\mathbb{G}_m^2$.
By Lemma~\ref{lem:subgrp-glnz}, we conclude that there exists an abelian characteristic subgroup $A\leqslant G$ with index bounded by $l:=l(2,Jk!)$, such that $A<\mathbb{G}_m^2$.
Thus, it suffices to pick $N\geq \max\{ l,4MJk!\}$.
Note that this value only depends on $M$.
\end{proof}

\begin{proposition}\label{prop:cyclic-inv-del-pezzo}
Let $M$ be a positive integer.
There exists a positive integer $N:=N(M)$,
only depending on $M$, satisfying the following.
Let $X$ be a del Pezzo surface and $G\leqslant {\rm Aut}(X)$ a finite sugroup with
$G\simeq \zz_m$ and $m\geq N$.
Let $B$ be an effective divisor on $X$ satisfying the following conditions:
\begin{enumerate}
    \item $(X,B)$ is $G$-invariant
    \item $(X,B)$ is log canonical, and
    \item $M(K_X+B)\sim 0$.
\end{enumerate}
Then $X$ admits a $\mathbb{G}_m$-action so that
$(X,B)$ is $\mathbb{G}_m$-invariant.
Furthermore, there is a subgroup $A\leqslant G$
with index bounded by $N$, such that 
$A\leqslant \mathbb{G}_m \leq {\rm Aut}(X,B)$.
\end{proposition}

\begin{proof}
From the classification of automorphisms of del Pezzo surfaces (see, e.g.,~\cites{Dol12,Hos96, MH74}), 
we know that $X$ must be the blow-up of $\pp^2$ at up to four points
with no three of them collinear.
Otherwise, the automorphism group of $X$ is finite.
These surfaces have a bounded number of $(-1)$-curves and every automorphism will act as a permutation group of such curves.
Let $S$ be the set of such curves and $k:=|S|$.
Hence, we have an exact sequence
\[
1\rightarrow A \rightarrow G \rightarrow S_k.
\]
Note that $A$ fixes all the curves of the set $S$.
Hence, $A$ acts as an automorphism group of $(\pp^2,B_{\pp^2})$, where $B_{\pp^2}$ is the push-forward of $B$ to $\pp^2$.
Furthermore, $A$ leaves the images of the $(-1)$-curves which are contracted to $\pp^2$ fixed.
Up to an isomorphism of $\pp^2$, we may assume that $A$ is acting as the multiplication by a root of unity on 
a one-dimensional torus of the big torus $\mathbb{G}_m^2\subset \pp^2$.
Note that $|A|\geq m/k!$.

We claim that $B_{\pp^2}$ is $\mathbb{G}_m$-invariant for such one-dimensional torus.
Indeed, any curves $C\subset \pp^2$ which is $A$-invariant but is not $\mathbb{G}_m$-invariant has degree at least $m/k!$.
Since $C$ appears on $B_{\pp^2}$ with coefficient at least $\frac{1}{M}$, we conclude that for $m\geq 4Mk!$, the boundary $B_{\pp^2}$ is $\mathbb{G}_m$-invariant.
Note that the $A$-action leaves a point $p\in \pp^2$ and a line $L\subset \pp^2$ invariant.
Hence, the centers of $X\rightarrow \pp^2$ must be contained in $L\cup\{p\}$.
Furthermore, $L\cup \{p\}$ is also $\mathbb{G}_m$-invariant.
We conclude that the $\mathbb{G}_m$-action lifts to $X$ in such a way $(X,B)$ is $\mathbb{G}_m$-invariant.
Hence, it suffices to pick $N\geq 4Mk!$.
\end{proof}

\subsection{Locally toric surface morphisms}

In this subsection, we recall an invariant
called the complexity. 
We also recall a characterization of toric morphisms using the complexity. 
First, we will introduce the complexity of a log pair with respect to a contraction.
In~\cite{BMSZ18}, the authors introduce a more general version of the complexity.
For our purposes, the following definition suffices.

\begin{definition}{\em
A {\em decomposition} of an effective divisor $B$ is an expression of the form $B:=\sum_{i=1}^k b_i B_i$, where the $B_i$'s are effective Weil divisors and the $b_i$'s are non-negative real numbers. We denote by $|\Sigma|:=\sum_{i=1}^k b_i$ the {\em norm} of the decomposition.
Each effective divisor $B$ on an algebraic variety $X$ has a natural decomposition given by the prime components of the support of $B$.

Let $(X,B)$ be a log canonical pair and
$X\rightarrow C$ be a contraction.
Let $\Sigma$ be a decomposition of $B$.
Let $c\in C$ be a closed point.
The {\em complexity} of $(X,B)$ over $c$, denoted by
$\gamma_c(X,B)$, is defined to be
\[
\dim(X)+\rho_c(X/C)-|\Sigma|.
\]
Here, $\rho_c(X/C)$ is the relative Picard rank of $X$ over $c\in C$.
The {\em local complexity} of a log pair $(X,B)$ at a closed point $x\in X$ 
is just the complexity over $x$ with respect to the identity morphism.
}
\end{definition}

We recall some results about the complexity log log Calabi--Yau pairs.

\begin{theorem}{\rm(cf.~\cite{Kol92}*{Theorem 18.22})}\label{thm:local-comp}
Let $(X,B)$ be a log canonical pair and $x\in X$ a closed point.
Then $\gamma_x(X,B)\geq 0$.
Furthermore, if $\gamma_x(X,B)=0$,
then $(X,\lfloor B\rfloor)$ is locally toric.
\end{theorem}

\begin{theorem}
{\rm(cf.~\cite{Sho00}*{Theorem 6.4})}\label{thm:locally-toric-surface}
Let $X$ be a surface.
Let $(X,B)$ be a log Calabi--Yau pair and let $X\rightarrow C$ be a fibration to a curve.
Let $c \in C$ be a closed point.
Then, $\gamma_c(X,B)\geq 0$.
Furthermore, if
$\gamma_c(X,B)=0$, 
then $(X,\lfloor B \rfloor)$ is locally toric over $c\in C$.
\end{theorem}

\begin{theorem}{\rm(cf.~\cite{BMSZ18}*{Theorem 1.2})}\label{thm:char-toric-comp}
Let $(X,B)$ be a log Calabi--Yau projective pair. 
Then $\gamma(X,B)\geq 0$.
Furthermore, if
$\gamma(X,B)=0$, then $(X,\lfloor B\rfloor)$ is isomorphic to a toric log pair.
\end{theorem}

\begin{proposition}\label{prop:locally-toric-morphism}
Let $(X,B)$ be a log Calabi--Yau surface.
Let $X\rightarrow C$ be a fibration to a curve.
Assume that $\lfloor B\rfloor_{\rm hor}$ has two irreducible components over $c\in C$.
Then, one of the following statement holds:
\begin{enumerate}
    \item If ${\rm lct}_c((X,B);\pi^*(c))>0$, 
    then $(X,B)$ is locally toric over $c\in C$.
    \item If ${\rm lct}_c((X,B);\pi^*(c))=0$, then there is a crepant birational transformation of $(X,B)$ over $c$, which is an isomorphism outside $\pi^{-1}(c)$, and is locally toric over $c\in C$
\end{enumerate}
\end{proposition}

\begin{proof}
Assume that ${\rm lct}_c((X,B);\pi^*(c))=t>0$.
Since this log canonical threshold is positive, 
then the two irreducible components of
$\lfloor B\rfloor_{\rm hor}$ are disjoint over $c\in C$.
Otherwise, $(X,B)$ is strictly log canonical over $c$.
Consider the log pair $(X,B+t\pi^*(c))$ which is log Calabi--Yau
(up to shrinking around $c$).
Let $(Y,B_Y)$ be a dlt modification 
of $(X,B+t\pi^*(C))$ over $C$.
Then, $B_Y$ contains a vertical curve with coefficient one mapping to $c$.
We call such a curve $E$.
Note that ${\rm coeff}_P(B_Y)>0$
for every prime divisor $P$ mapping to $c$.
Let $(Z,B_Z-\epsilon E_Z)$ be the outcome of running a minimal model program
for $(Y,B_Y-\epsilon E)$ over $C$,
where $\epsilon$ is a small positive rational number.
Here, as usual,
$B_Z$ is the push-forward of $B_Y$ to $Z$
and $E_Z$ is the push-forward of $E$ to $Z$.
Note that $(Z,B_Z)$ is a $\qq$-factorial log Calabi--Yau pair with a fibration to $C$.
Furthermore, by construction, we have that
$\rho_c(Z/C)=1$.
We conclude that 
$\gamma_c(Z,B_Z)=0$.
By Theorem~\ref{thm:locally-toric-surface},
we conclude that $(Z,B_Z)$ is locally toric over $c\in C$.
Note that the morphism $Y\rightarrow Z$ only extract prime divisors with log discrepancy 
in the interval $[0,1)$ with respect to the $(Z,B_Z)$.
We conclude that $(Y,B_Y)$ is locally toric over $c\in C$ as well.
Finally, since we have a contraction $Y\rightarrow X$ over $c\in C$, 
we conclude that $(X,B)$ is locally toric over $c$.

Assume that ${\rm lct}_c((X,B);\pi^*(c))=0$.
We can apply the connectedness of the log canonical centers~\cite{HH19}*{Theorem 1.2}. We conclude that either the two irreducible components of $\lfloor B\rfloor_{\rm hor}$ are connected by components of $\lfloor B\rfloor_{\rm ver}$
or the two components of $\lfloor B\rfloor_{\rm hor}$ intersect over $c$.
If the former case happens, we take a component $E\subset \lfloor B\rfloor_{\rm ver}$ and run a minimal model for $(X,B-\epsilon E)$ over $C$.
It terminates with a model $(Y,B_Y-\epsilon E_Y)$ for which
$\gamma_c(Y,B_Y)=0$.
By Theorem~\ref{thm:locally-toric-surface}, we have that $(Y,B_Y)$ is locally toric over $c\in C$.
On the other hand, if the latter case happens, we take a $\qq$-factorial dlt modification $(Y,B_Y)$ of $(X,B)$ over $c$ in $C$.
Let $E$ be a prime divisor of $\lfloor B_Y\rfloor$ which maps to $c$.
We run a minimal model program for
$(Y,B_Y-\epsilon E)$ over $C$ which terminates with a minimal model $(Z,B_Z-\epsilon E_Z)$.
In this model, the effective divisor $B_Z$ is reduced.
Note that $\gamma_c(Z,B_Z)=0$. Hence,
we conclude that $(Z,B_Z)$ is locally toric over $c\in C$.
Observe that in any case, the birational map
$X\dashrightarrow Z$ is an isomorphism over $C\setminus c$.
\end{proof}

\subsection{$G$-equivariant locally toric surface morphisms}

In this subsection, we introduce a $G$-equivariant version of the complexity. 
We prove the $G$-equivariant versions of some of the results in the last subsection.

\begin{definition}
{\em 
Let $X$ be a normal quasi-projective variety and $G\leqslant {\rm Aut}(X)$ be a finite automorphism group.
We say that an effective Weil divisor $D$ on $X$ is {\em $G$-prime} or {\em $G$-irreducible} if it $G$-invariant and is not the union of two $G$-invariant effective Weil divisors.
For instance, the $G$-orbit of a prime reduced divisor on $X$ is $G$-prime.
}
\end{definition}

\begin{definition}{\em 
Let $X$ be a quasi-projective normal variety and 
$G\leqslant {\rm Aut}(X)$ be a finite automorphism group.
Let $\phi\colon X\rightarrow Z$ be a $G$-equivariant fibration.
We have an exact sequence 
$1\rightarrow G_{\phi,f} \rightarrow G\rightarrow G_{\phi,b}\rightarrow 1$. Here, $G_{\phi,b}$ is the finite automorphism group induced on the base and $G_{\phi,f}$ is the finite automorphism group acting fiber-wise.
Analogously, $G_{\phi,f}$ is the kernel of the natural surjective homomorphism $G\rightarrow G_{\phi,b}$.
We call $G_{\phi,f}$ (resp. $G_{\phi,b}$) the {\em fiber-wise automorphism group} (resp. {\em base automorphism group}).
If $\phi$ is clear from the context, we will omit it from the notation.
}
\end{definition}

\begin{definition}{\em
Let $X$ be an algebraic variety and
$G\leqslant {\rm Aut}(X)$ a finite subgroup.
A {\em $G$-invariant decomposition} of an effective $G$-invariant divisor $B$ is an expression of the form $B:=\sum_{i=1}^k b_i B_i$, where the $B_i$'s are effective $G$-invariant Weil divisors and the $b_i$'s are non-negative real numbers. We denote by $|\Sigma|^G:=\sum_{i=1}^k b_i$
the {\em $G$-invariant norm} of the decomposition.
Each $G$-invariant effective divisor $B$ on an algebraic variety $X$ has a natural decomposition given by the $G$-orbits of the prime decomposition of $B$.

Let $(X,B)$ be a $G$-invariant log canonical pair and
$X\rightarrow C$ be a $G$-equivaraint contraction.
Let $\Sigma$ be a $G$-invariant decomposition of $B$.
Let $c\in C$ be a closed point.
The {\em $G$-complexity} or {\em $G$-invariant complexity} of $(X,B)$ over $c$, denoted by
$\gamma^G_c(X,B)$, is defined to be
\[
\dim(X)+\rho^G_c(X/C)-|\Sigma|^G.
\]
Here $\rho^G_c(X/C)$ is the $G$-invariant relative Picard rank of $X$ over $c\in C$.
}
\end{definition}

\begin{theorem}\label{thm:G-locally-toric-surface}
Let $X$ be a surface and
$G\leqslant {\rm Aut}(X)$ be a finite automorphism group.
Let $(X,B)$ be a $G$-invariant log Calabi--Yau pair 
and $X\rightarrow C$ be a $G$-equivariant fibration to a curve.
Let $c\in C$ be a closed point.
Then, $\gamma^G_c(X,B)\geq 0$.
Furthermore, if 
$\gamma^G_c(X,B)=0$, then $(X,\lfloor B\rfloor)$ is $G$-equivariantly locally toric over the $G_b$-orbit of $c$ in $C$.
\end{theorem}

\begin{proof}
Quotienting $X$ by $G$, we obtain a commutative diagram as follows:
\[
 \xymatrix{
(X,B) \ar[r]^-{/G} \ar[d]_-{\pi} & (Y,B_Y) \ar[d]^-{\pi_Y} \\
C \ar[r]^-{/G_b} & C_Y.\\ 
 }
\]
Let $c_0$ be the image of the $G_b$-orbit of $c$ in $C$.
We claim that, we have the following inequalities
\[
\gamma^G_c(X,B)\geq \gamma_{c_0}(Y,B_Y)\geq 0.
\]
Indeed, $\dim(X)=\dim(Y)$, 
$\rho^G_c(C/X)=\rho_{c_0}(Y/C_Y)$, and 
\begin{equation}\label{quotient-coefficients}
|\Sigma|^G =\sum_{i=1}^k b_i 
\leq \sum_{i=1}^k \left( 1-\frac{1}{m}+\frac{b_i}{m}
\right)
=|\Sigma|,
\end{equation}
where $m$ is the multiplicity index of $X\rightarrow Y$ at a $G$-prime component $B_i$ of coefficient $b_i$.
Note that the equality of~\eqref{quotient-coefficients} holds if and only if 
$b_i=1$ or $m_i=1$.
Hence, the inequality
$\gamma^G_c(X,B)\geq 0$ holds.
Furthermore, the equality
$\gamma^G_c(X,B)=\gamma_{c_0}(Y,B_Y)$ holds if and only if $X\rightarrow Y$ only ramifies at divisors of $\lfloor B\rfloor$.
Assume that $\gamma^G_c(X,B)=0$, 
then we have that
$\gamma_{c_0}(Y,B_Y)=0$.
By Theorem~\ref{thm:locally-toric-surface}, we conclude that $(Y,\lfloor B_Y\rfloor)$ is toric over $c_0$.
Furthermore, $X\rightarrow Y$ only ramifies at divisors of $\lfloor B_Y\rfloor$.
Let $D_Y$ be the toric boundary of $Y$ over $c_0$.
In particular, we have that
$\lfloor B_Y\rfloor\subset D_Y$.
Note that the complement $U_Y$ of $D_Y$ on $Y$ has Picard rank zero.
Let $(X,D_X)$ be the log pull-back of $(Y,D_Y)$ to $X$.
By shrinking around the $G_b$-orbit of $c$, we may assume that the prime components of $D_X$ generate the relative Picard of $X$ over $C$.
Let $k$ be the number of prime components of $D_X$.
We have two linearly independent algebraic relations among the irreducible components of $D_X$.
These two relations come from pulling back the corresponding relations on $Y$ over $C$.
Thus, $\rho_c(X/C)\leq k-2$.
We conclude that
\[
\gamma_c(X,D_X)=\dim(X)+\rho_c(X/C)-|D_X| 
\leq 2 + (k-2) - k \leq 0.
\]
Thus, the Picard rank of $X$ over the $G_b$-orbit of $c$ is just $k-2$ and $(X,D_X)$ is locally a log Calabi--Yau toric pair over $c$.
Since $\lfloor B\rfloor \subset D_X$, we conclude that $(X,\lfloor B\rfloor)$ is locally toric over $c$.
\end{proof}

\begin{theorem}\label{thm:g-inv-comp}
Let $X$ be a normal projective variety of dimension $n$
and $G\leqslant {\rm Aut}(X)$ be a finite automorphism group.
Let $(X,B)$ be a $G$-invariant log Calabi--Yau pair.
Then $\gamma^G(X,B)\geq 0$.
Furthermore, if $\gamma^G(X,B)=0$, then
$(X,\lfloor B\rfloor)$ is isomorphic to a toric pair.
\end{theorem}

\begin{proof}
Let $(Y,B_Y)$ be the quotient of $(X,B)$ by $G$.
We know that 
$\gamma^G(X,B)\geq \gamma(Y,B_Y)\geq 0$.
This proves the first statement.
Moreover, equality holds if and only if $X\rightarrow Y$ only ramifies at $\lfloor B\rfloor$.
Assume that $\gamma^G(X,B)=0$,
then $(Y,\lfloor B_Y\rfloor)$ is a toric pair.
Let $(Y,D_Y)$ be the toric log Calabi--Yau structure.
Denote by $(X,D_X)$ the log pull-back of $(Y,D_Y)$ to $X$.
Note that $X\rightarrow Y$ is unramified over $\mathbb{G}_m^n$.
Hence, $X\setminus D_X\simeq \mathbb{G}_m^n$.
We conclude that the Picard group of $X$ is generated by the irreducible components of $D_X$.
Let $k$ be the number of such irreducible components.
Then, $\rho(X)\leq k$.
We claim that there are at least $n$ linearly independent algebraic relations among the irreducible components of $D_X$.
Indeed, this statement holds among the prime components of $D_Y$ in $Y$. Pulling back these linear relations we obtain the corresponding relations on $X$.
Thus, $\rho(X)\leq k-n$,
where $n$ is the dimension of $X$.
Now, we can proceed to compute the complexity of $(X,D_X)$.
Note that
\[
\gamma(X,D_X)=\rho(X)+\dim(X)-|D_X|
\leq k-n+n-k=0.
\]
By Theorem~\ref{thm:char-toric-comp}, we conclude that $\rho(X)=k-n$
and $(X,D_X)$ is a toric log Calabi--Yau pair.
Since $\lfloor B \rfloor \subset D_X$, we conclude that
$(X,\lfloor B\rfloor)$ is a toric pair as well.
\end{proof}

\begin{proposition}\label{prop:G-locally-toric-morphism}
Let $X$ be a surface and
$G\leqslant {\rm Aut}(X)$ be a finite automorphism group.
Let $(X,B)$ be a $G$-invariant log Calabi--Yau pair and $X\rightarrow C$ 
be a $G$-equivariant fibration to a curve.
Assume that $\lfloor B\rfloor_{\rm hor}$ has two irreducible components over $c\in C$.
Then, one of the following statements hold:
\begin{enumerate}
    \item If ${\rm lct}((X,B);\pi^*(c))>0$, then
    $(X,B)$ is $G$-equivariantly locally toric over the $G_b$-orbit of $c$ in $C$.
    \item If ${\rm lct}((X,B);\pi^*(c))=0$, then there is a crepant $G$-equivariant birational transformation of $(X,B)$ over the $G_b$-orbit of $c$ in $C$ which is an isomorphism outisde
    the $G$-orbit of $\pi^{-1}(c)$ and is $G$-equivariantly locally toric over the $G_b$-orbit of $c$ in $C$.
\end{enumerate}
\end{proposition}

\begin{proof}
Assume that 
${\rm lct}_c((X,B);\pi^*(c))=t>0$.
Denote by $G_b c$ the $G_b$-orbit of $c$ in $C$.
Since this log canonical threshold is positive, then
the two irreducible components of $\lfloor B\rfloor_{\rm hor}$ are disjoint over $G_b c$.
Furthermore, the log pair
$(X,B+t\pi^*(G_b c))$ is a $G$-equivariant strictly log canonical pair.
This pair is log Calabi--Yau (up to shrinking around $G_b c$).
Let $(Y,B_Y)$ be a $G$-equivariant dlt modification of $(Y,B+t\pi^*(G_b c))$ over $C$.
Then, $B_Y$ contains a $G$-prime vertical curve with coefficient one mapping to $G_bc$.
We call such curve $E$.
Note that $E$ may be reducible.
Furthermore, 
${\rm coeff}_P(B_Y)>0$, for every prime divisor $P$ mapping to $c$.
We run a $G$-equivariant minimal model program for $(Y,B_Y-\epsilon E)$ over $C$,
where $\epsilon$ is a positive small rational number.
This minimal model program terminates with a $G$-equivariant model $(Z,B_Z-\epsilon E_Z)$ over $C$.
Here, $B_Z$ is the push-forward of $B_Y$ to $Z$ and $E_Z$ is the push-forward of $E$ to $Z$.
Note that $(Z,B_Z)$ is a $\qq$-factorial $G$-equivariant log Calabi--Yau pair with a $G$-equivariant fibration to $C$.
Furthermore, by construction, we have that
$\rho^G_c(Z/C)=1$.
Furthermore, we have that $\gamma_c(Z,B_Z)=0$.
By Proposition~\ref{prop:locally-toric-morphism}, we conclude that $(Z,B_Z)$ is $G$-equivariantly locally toric over $G_b c$.
Note that the morphism $Y\rightarrow Z$ is $G$-equivariant and only extract prime divisors with log discrepancy in the interval $[0,1)$ with respect to $(Z,B_Z)$. We conclude that
$(Y,B_Y)$ is $G$-equivariantly locally toric over $G_b c$ in $C$.
Finally, since $Y\rightarrow X$ is a $G$-equivariant contraction over $G_b c$ in $C$, we conclude that
$(X,B)$ is $G$-equivariantly locally toric over the $G_b$-orbit of $c$ in $C$.
Indeed, $Y\rightarrow X$ is just the morphism over $C$ defined by a semiample $G$-invariant toric divisor on $Y$.

Assume that ${\rm lct}_c((X,B);\pi^*(c))=0$.
We can apply the connectedness of the log canonical centers~\cite{HH19}*{Theorem 1.2}. 
We conclude that
either the components of $\lfloor B\rfloor_{\rm hor}$ are connected 
by components of $\lfloor B\rfloor_{\rm ver}$
or the two components of $\lfloor B\rfloor_{\rm hor}$ intersect over $c$.
If the former case happens, then
we can take a $G$-prime effective divisor $E\subset \lfloor B\rfloor_{\rm vert}$ and run a $G$-equivariant minimal model program for $(X,B-\epsilon E)$ over $C$.
It terminates with a model $(Y,B_Y-\epsilon E_Y)$ for which $\gamma_c(Y,B_Y)=0$.
By Proposition~\ref{prop:locally-toric-morphism},
we have that $(Y,B_Y)$ is $G$-equivariantly locally toric over $c$.
On the other hand, if the latter case happens, we take a $G$-equivariant $\qq$-factorial dlt modification $(Y,B_Y)$ of $(X,B)$ over $c\in C$.
Let $E$ be an effective $G$-prime divisor contained in $\lfloor B_Y\rfloor$ which maps to $c$.
We run a $G$-equivariant minimal model program for $(Y,B_Y-\epsilon E)$ over $C$ which terminates with a $G$-equivariant minimal model $(Z,B_Z-\epsilon E_Z)$.
In this $G$-equivariant minimal model, the $G$-invariant divisor $B_Z$ is reduced.
Note that $\gamma_c(Z,B_Z)=0$. Hence,
we conclude that $(Z,B_Z)$ is locally toric over $c\in C$.
Observe that in any case the $G$-equivariant birational map
$X\dashrightarrow Z$ is an isomorphism
over $C\setminus G_b c$.
\end{proof}

\begin{corollary}\label{cor:G-equivariant-birational-locally-toric}
Let $X$ be a surface and $G\leqslant {\rm Aut}(X)$ be a finite automorphism group.
Let $(X,B)$ be a $G$-equivariant log Calabi--Yau pair and $X\rightarrow C$ be a $G$-equivariant fibration to a curve.
Assume that $\lfloor B\rfloor_{\rm hor}$ has two irreducible components over $C$.
Then, there is a crepant $G$-equivariant birational map $X\dashrightarrow Y$ over $C$, such that
$(Y,B_Y)$ is $G$-equivariantly locally toric over $C$.
\end{corollary}

\begin{proof}
Let $c\in C$ be a point so that
${\rm lct}((X,B);\pi^*(c))>0$, then
$(X,B)$ is $G$-equivariantly locally toric over the $G_b$-orbit of $c$ in $C$ by Proposition~\ref{prop:G-locally-toric-morphism}.
We conclude that $(X,B)$ is $G$-equivariantly locally toric over the open set $U$ of $C$ on which the above log canonical threshold is 
positive.
Assume that there is a unique closed point $c_0$ in $C$ over whose orbit the above log canonical threshold is zero.
 
We have that ${\rm lct}((X,B);\pi^*(c_0))=0$. 
By Proposition~\ref{prop:G-locally-toric-morphism},
there is a crepant $G$-equivariant birational transformation $(Y,B_Y)$ of $(X,B)$ over the $G_b$-orbit of $c_0$ in $C$, which is an isomorphism outside the $G$-orbit of $\pi^{-1}(c)$, and
is $G$-equivariantly locally toric over
the $G_b$-orbit of $c$ in $C$.
Note that after this birational transformation of $(X,B)$, the log pair $(Y,B_Y)$ remains 
$G$-equivariantly locally toric over $U$.
Hence, $Y\rightarrow C$ is $G$-equivariantly locally toric over $C$
for the log pair $(Y,B_Y)$.
In the case that $C$ has several closed points over which the log canonical threshold is zero, we proceed inductively.
\end{proof}

\subsection{Surface fibrations with large fiber-wise automorphisms}
\label{subsec:large-fiber}

In this subsection, we study $G$-equivariant fibrations of Fano type surfaces such that the induced 
automorphism on a general fiber is large.

\begin{proposition}\label{prop:large-fiber}
Let $M$ be a positive integer.
There exists a positive integer $N:=N(M)$, only depending on $M$,
satisfying the following.
Assume that the following conditions hold:
\begin{enumerate}
    \item $X$ is a Fano type surface
    and $G\leqslant {\rm Aut}(X)$ is a finite automorphism group, 
    \item $(X,B)$ is a $G$-invariant log Calabi--Yau pair such that $M(K_X+B)\sim 0$,  
    \item $X\rightarrow C$ is a $G$-equivariant fibration to a curve, and
    \item $G_f$ is abelian with $|G_f|\geq N$.
\end{enumerate}
Then, up to passing to a $G$-equivariant crepant birational model of $(X,B)$, we may assume that $\lfloor B\rfloor_{\rm hor}$ has two prime components.
\end{proposition}

\begin{proof}
First, we reduce to the case in which $B_{\rm hor}$ is a reduced divisor.
We denote by $(F,B_F)$ the restriction of $(X,B)$ to a general fiber $F$.
Note that $M(K_F+B_F)\sim 0$.
By Lemma~\ref{lem:ab-action-p1}, we may assume that $G_f$ is a cyclic group acting as the multiplication by a root of unity on $F\simeq \pp^1$.
Note that the coefficients of $B_F$ are larger than $\frac{1}{2M}$.
Hence, we conclude that for $N\geq 4M$, the pair
$(F,B_F)$ is isomorphic to
$(\pp^1,\{0\}+\{\infty\})$.
Thus, all horizontal components of $B$ are reduced.
Furthermore, $B$ has either one or two horizontal components.

We reduce to the case in which 
$C_0:=B_{\rm hor}$ is normal.
Assume that $C_0$ is not a normal curve.
Since $C_0$ must be a semi-log canonical curve, then it has a node $c\in C_0$.
Note that $c_0 \in X$ is a log canonical center of $(X,B)$ by inversion of adjunction.
Let $c$ be the image of $c_0$ in $C$.
We can extract a $G$-prime log canonical $E$ center over the $G$-orbit of $c_0$ in $X$.
We call $Y\rightarrow X$ such $G$-equivariant birational morphism.
We run a $G$-equivariant minimal model program for
$(Y,B_Y-E)$ over $C$.
We get a model $(Z,B_Z)$.
Note that all the hypotheses on $(X,B)$ are 
satisfied by $(Z,B_Z)$.
Furthermore, the curve $B_{Z,{\rm hor}}$ is normal over $C\setminus c$ and is normal on a neighborhood of the fiber over $c$.
Hence, $B_{Z,{\rm hor}}$ is a normal curve.
Replacing the pair $(X,B)$ with $(Z,B_Z)$, we may assume that
$C_0:=\lfloor B\rfloor_{\rm hor}$ is a normal curve.

Note that we have a finite morphism 
$C_0\rightarrow C$.
We prove that the finite morphism $C_0\rightarrow C$ must be unramified. 
Assume that $C_0\rightarrow C$ ramifies over a point $c\in C$.
Let $c_0$ be the pre-image of $c$ in $C_0$.
The multiplicity index of $C_0\rightarrow C$ at $c$ equals two. 
Otherwise, the boundary $B_F$ has three or more points with coefficient one, leading to a contradiction.
Let $Y\rightarrow X$ be the normalization of the base change induced by $C_0\rightarrow C$.
Hence, over the $G_b$-orbit of $c$ in $C$ the finite morphism $Y\rightarrow X$ is the quotient by an involution $\tau$ on $Y$.
Let $(Y,B_Y)$ be the log pull-back of $(X,B)$ to $Y$.

We claim that $G_f=\langle g_f\rangle$ acts as a group of automorphisms on $(Y,B_Y)$ so that $g_f \tau g_f^{-1}\in \langle \tau \rangle$.
We call $\pi \colon X\rightarrow C$.
Indeed, we have a commutative diagram as follows:
\[
\xymatrix{
X \ar[r]^-{\pi} \ar[d]^-{g_f} & C \ar[d]^-{{\rm id}_C} & C_Y\ar[d]^-{{\rm id}_{C_Y}}
\ar[l]_-{/\tau} \\
X\ar[r]^-{\pi} & 
C & C_Y\ar[l]_-{/\tau}.
}
\]
Hence, we have an induced automorphism $g_f\in {\rm Aut}(Y)$ acting as the multiplication by a root of unity on a general fiber.
Furthermore, we have a commutative diagram:
\[
 \xymatrix{
 Y\ar[d]_-{/\tau}\ar[r]^-{g_f} & Y\ar[d]^-{/\tau} \\ 
 X\ar[r]^-{g_f} & X.}
\]
By abuse of notation, we are calling both automorphisms $g_f$.
Thus, we conclude that the following diagram commutes:
\[
 \xymatrix{
 Y\ar[d]_-{/\tau}\ar[r]^-{g_f^{-1}} & Y\ar[d]^-{/\tau}\ar[r]^-{\tau} & Y\ar[d]^-{/\tau}\ar[r]^-{g_f}& Y\ar[d]^-{/\tau} \\ 
 X\ar[r]^-{g_f^{-1}} & X\ar[r]^-{{\rm id}_X} & X \ar[r]^-{g_f} & X.\\
 }
\]
Since the automorphism on $X$ obtained by composing the bottom horizontal arrows is ${\rm id}_X$, we conclude that the automorphism on $Y$ obtained by composing the upper horizontal arrows is contained in $\langle \tau\rangle$. This proves the claim.

We denote by $G_Y:=\langle g_f,\tau \rangle$.
Note that $B_Y$ has two horizontal irreducible components.
By Proposition~\ref{prop:G-locally-toric-morphism}, we know that $(Y,B_Y)$ is $G_f$-equivariantly locally toric over $c_0$.
Observe that the toric boundary of $Y$ is $G_Y$-invariant, hence $G_Y$ acts as an automorphism group of the complement of the toric boundary.
The complement of the toric boundary, locally over $c_0$, is isomorphic to $D_0\times \mathbb{G}_m$ (where $D_0$ is the punctured disk).
In the coordinates $(t_1,t_2)$ of $D_0 \times \mathbb{G}_m$ the automorphism $\tau $ acts as $(-t_1,\lambda t_2^{-1})$, where $\lambda \neq 0$.
On the other hand, $g_f$ acts as
$(t_1,\mu_{|G_f|}t_2)$ where
$\mu_{|G_f|}$ is a $|G_f|$-root of unity. Thus,
the relation $g_f\tau g_f^{-1}=\langle \tau \rangle$ holds if and only if $|g_f|=|G_f|\leq 2$.
We deduce that for $|G_f|>2$ the finite map $C_0\rightarrow C$ must be unramified.
Since $C$ is isomorphic to $\pp^1$, we conclude that $C_0$ must have two connected components.
\end{proof}

\subsection{Surface fibrations with large base automorphisms}
\label{subsec:large-base}

In this subsection, we study $G$-equivariant fibrations of Fano type surfaces such that the induced automorphism on the base is large.

\begin{proposition}\label{prop:large-base}
Let $M$ be a positive integer.
There exists a positive integer $N:=N(M)$, only depending on $N$,
satisfying the following.
Assume that the following conditions hold:
\begin{enumerate}
\item $X$ is a Fano type surface, 
$G\leqslant {\rm Aut}(X)$ is a finite automorphism group, and $\rho^G(X)=2$,
\item $(X,B)$ is a $G$-invariant log Calabi--Yau pair such that $M(K_X+B)\sim 0$, 
\item $X\rightarrow C$ is a $G$-equivariant fibration to a curve,
\item $G_b$ is an abelian group with $|G_b|\geq N$, 
\item the fibers over $\{0\}$ and $\{\infty\}$ are $G$-prime components of $\lfloor B\rfloor$, and 
\item the second extremal ray of the $G$-equivariant cone of curves define a birational contraction.
\end{enumerate}
Then, $(X,B)$ satisfies one of the following:  
\begin{enumerate}
\item The pair $(X,B)$ is a $G$-equivariant log Calabi--Yau toric pair, or
\item the pair $(X,B)$ admits a $\mathbb{G}_m$-action,
with $H<\mathbb{G}_m\leqslant {\rm Aut}(X,B)$ and $H\leqslant G$ has index bounded by $N$
\end{enumerate}
Furthermore, if we assume that
$G_f$ is abelian with $|G_f|\geq N$, then the first case holds.
\end{proposition}

\begin{proof}
Let $F_0$ and $F_\infty$ be the fibers at zero and infinity for the morphism $X\rightarrow C$.
By assumption, $F_0$ and $F_\infty$ are $G$-prime components of $\lfloor B\rfloor$.
Let $X\rightarrow X_0$ be the $G$-equivariant morphism defined by the second extremal ray of the $G$-equivariant cone of curves.
We are assuming it is a birational morphism.
Note that every irreducible component of the exceptional locus of $X\rightarrow X_0$ is horizontal over $C$.
Hence, such components intersect both $F_0$ and $F_\infty$.
We can apply the connectedness of log canonical centers. We conclude that all fibers of $X\rightarrow X_0$ appear in $\lfloor B \rfloor$.
Note that $X\rightarrow X_0$ has at most two irreducible components contained in its exceptional locus.
Furthermore, each such an exceptional curve intersects $F_0$ and $F_\infty$ exactly once.
Otherwise, the induced pair $(X_0,B_0)$ would have negative local complexity at the image of such an exceptional curve.

Consider the following diagram obtained by quotienting by the finite group action:
\[
 \xymatrix{
(X,B) \ar[r]^-{/G} \ar[d]_-{\pi} & (Y,B_Y) \ar[d]^-{\pi_Y} \\
C \ar[r]^-{/G_b} & C_Y. \\ 
 }
\]
Let $(C_Y,B_Y)$ be the pair obtained by the canonical bundle formula.
By assumption, $(C_Y,B_Y)\simeq (\pp^1,\{0\}+\{\infty\})$ and
$\rho(Y)=2$.
Furthermore, all fibers of $\pi_Y$ are irreducible.
We will consider three different cases depending on the number and coefficients of the horizontal components of $B$.

\textit{Case 1:} In this case, we assume that $\lfloor B \rfloor$ has a unique horizontal prime component which is unramified over $C$. 
Then, $B$ has at least one other horizontal component. 
Furthermore, $\lfloor B\rfloor_{\rm hor}$ is $G$-invariant.
Note that $F_0,F_\infty$ and $\lfloor B\rfloor_{\rm hor}$ are $G$-invariant divisors. Then, we have that
$\gamma^G(X,B)<1$ and 
$\gamma(Y,B_Y)<1$.
By~\cite{BMSZ18}*{Theorem 1.2}, we conclude that $(Y,B_Y)$ is a toric pair.
Observe that the finite morphism $X\rightarrow Y$ ramifies over a unique horizontal prime divisor $P$ on $Y\setminus \lfloor B_Y \rfloor$.
Otherwise, we would have that $\gamma(Y,B_Y)<0$.
Indeed, any ramification divisor will appear with coefficient at least $\frac{1}{2}$.
Let $P_Y$ be the image of $P$ on $Y$.
Note that $P_Y$ does not intersect $\lfloor B_Y\rfloor_{\rm hor}$,
otherwise the induced contraction $Y\rightarrow Y_0$ has negative local complexity at the image of $\lfloor B_Y\rfloor_{\rm hor}$.
We conclude that $P_Y$ is the closure of $\lambda \times \mathbb{G}_m \subset \mathbb{A}^1\times \mathbb{G}_m \simeq Y\setminus \lfloor B_Y\rfloor$ on $Y$.
In particular, the log pair 
$(Y,\lfloor B_Y\rfloor+\delta P_Y)$ is log canonical for every $0\leq \delta\leq 1$.

Hence, the pair 
$(X,\lfloor B \rfloor +\epsilon P)$ is anti-nef, log canonical, and $G$-invariant.
Indeed, the push-forward of this pair to $Y$ has the form:
\[
\left(Y,\lfloor B_Y \rfloor + \left(1-\frac{1-\epsilon}{m} \right)P_Y\right).
\]
Here, $m$ is the multiplicity index at $P$.
The above pair has complexity strictly less than $\frac{1}{2}$.
By~\cite{BMSZ18}*{Theorem 1.2}, we conclude that $P_Y$ is a torus invariant divisor.
Hence, $(Y,\lfloor B_Y\rfloor+P_Y)$ is a log Calabi--Yau pair.
We conclude that $(X,\lfloor B\rfloor+P)$
is a $G$-invariant log Calabi--Yau pair.
We set $D:=\lfloor B\rfloor+P$.
Then, $D$ has two $G$-invariant horizontal components.
By Theorem~\ref{thm:g-inv-comp}, we conclude that $(X,D)$ is a log Calabi--Yau toric pair.
If $|G_f|\geq N$, then we have that $(X,B)$ is a log Calabi--Yau toric pair as well.
Indeed, in such a case, the divisors $B$ and $D$ coincide.

We turn to prove the existence of the subgroup $H$.
Let $g_b$ be the generator of $G_b$.
Let $g_0\in G$ be an element whose image on $G_b$ is $g_b$.
Note that the divisors $\lfloor B\rfloor, P,F_0$ and $F_\infty$ are $G$-prime divisors.
Hence $g_0$ acts as 
$g_0\cdot (t_1,t_2) = (\mu^a t_1,\mu^b t_2)$  on the big torus $\mathbb{G}_m^2$.
Here, $\mu$ is a $n:=|G|$-root of unity
and $\mu^a$ is a $|G_b|$-root of unity.
We distinguish two possible cases.
Let $N_0$ as in the statement of Lemma~\ref{lem:ab-action-p1}.
If $n/{\rm gcd}(b,n)$ is larger than $N_0$, then $B=D$ 
and $(X,B)$ is a $G$-equivariant log Calabi--Yau toric pair.
Indeed, in such a case, the only possible horizontal divisors are $\lfloor B\rfloor$ and $P$.
If $n/{\rm gcd}(b,n)$ is bounded by $N_0$, then we can consider the subgroup $H:=\langle g_0^{n/{\rm gcd}(b,n)}
\rangle \leqslant G$.
Note that $H$ acts as $(\mu' t_1,t_2)$ on the big torus.
Furthermore, $H_b$ has bounded $H$ has bounded index on $G$.
Note that $H$ embeds naturally in the horizontal $\mathbb{G}_m$-action on the big torus given by $\lambda \cdot (t_1,t_2)=(\lambda t_1,t_2)$.
Is clear that in this case, we have $H\unlhd \mathbb{G}_mH$.
Note that no horizontal component of $B_{\rm hor}\setminus \lfloor B\rfloor_{\rm hor}$ intersect $\lfloor B\rfloor_{\rm hor}$ non-trivially.
Otherwise, the log pair $(X_0,B_0)$ has a point of negative local complexity.
We conclude that the boundary $B$ must be invariant with respect to the horizontal torus action.
Hence, we have that the log pair $(X,B)$ is $\mathbb{G}_m$-invariant. 

\textit{Case 2:} In this case, we assume that $B$ has a unique horizontal component.
Furthermore, $\lfloor B\rfloor$ has a unique horizontal component which admits a ramified morphism to $C$.
By Proposition~\ref{prop:large-fiber}, this case only happens if $|G_f|\leq N_0$.
Hence, $\lfloor B_Y\rfloor$ has a unique horizontal component $C_Z$ with a ramified morphism to $C_Y$.
Let $Z$ be the normalization of the main component of $Y\times_{C_Y} C_Z$.
Let $(Z,B_Z)$ be the log pull-back of $(Y,B_Y)$ to $Z$.
Note that $Y$ is the quotient of $Z$ by an involution $\tau_Z$.
We have a commutative diagram as follows:
\[
 \xymatrix{
(Z,B_Z) \ar[r]^-{/\tau_Z} \ar[d]_-{\pi_Z} & (Y,B_Y) \ar[d]^-{\pi_Y} \\
C_Z \ar[r]^-{/\tau} & C_Y. \\ 
 }
\]
Observe that all fibers of $Z\rightarrow C_Z$ are irreducible outside the points at zero and infinity.
Note that $\lfloor B_Z\rfloor$ has $k_0+k_\infty+2$ prime components.
Here, $k_0$ is the number of components over zero
and $k_\infty$ is the number of components over infinity. 
Furthermore, $\rho(Z)\leq k_0+k_\infty$.
By Theorem~\ref{thm:char-toric-comp}, we know  that $(Z,B_Z)$ is a log Calabi--Yau toric pair.
Let $W$ be the normalization of the main component of $X\times_Y Z$.
Note that $X$ is the quotient of $W$ by an involution $\tau_W$.
Let $(W,B_W)$ be the log pair obtained by pulling-back $(X,B)$ to $W$.
We obtain a diagram as follows:
\[
 \xymatrix{
(X,B)\ar[d]_-{/G} & (W,B_W)\ar[l]_-{/\tau_W} \ar[d] \\
(Y,B_Y)  & (Z,B_Z)\ar[l]_-{/\tau_Z}. \\ 
 }
\]
Note that $B_W$ is a reduced divisor.
Observe that $X\rightarrow Y$ is unramified over $Y\setminus \lfloor B_Y\rfloor$ 
and $Z\rightarrow Y$ is unramified over
$Y\setminus \lfloor B_Y\rfloor$.
We conclude that $W\rightarrow Z$ is unramified over $Z\setminus \lfloor B_Z\rfloor$.
Thus, the pair $(W,B_W)$ is a log Calabi--Yau toric pair.
Observe that $W$ is equipped with a $\tau_W$-invariant fibration to a curve $C_W$.
By construction, the induced log pair structure in $C_W$ is isomorphic to $(\pp^1,\{0\}+\{\infty\})$.

Let $G_0$ be a cyclic subgroup of $G$ surjecting onto $G_b$.
Let $g_0$ be the generator of $G_0$.
Note that $G_f$ and $G_0$ both act on $W$
as automorphism groups.
Let $(t_1,t_2)$ be the coordinates of the big torus of $W$.
The involution is acting as $\tau\cdot (t_1,t_2)=(-t_1,t_2^{-1})$.
Furthermore, $g_f$ acts as 
$g_f\cdot (t_1,t_2)=(t_1,\mu t_2^{\pm})$ and 
$g_0$ acts as 
$g_0 \cdot (t_1,t_2)=
(\eta^a t_1,\eta^b t_2^{\pm})$.
Here, as usual, $\lambda$ and $\eta$ are roots of unity.
Furthermore, the relation $g_0\tau_W g_0^{-1} \in \langle \tau_W\rangle$ holds.
We conclude that $g_0$ acts as 
$g_0\cdot (t_1,t_2)=(\eta^a t_1,\eta^b t_2^{\pm})$ with $\eta^{2b}=1$.
Then, there exists an index two subgroup $H_0$ of $G_0$ acting as $h_0\cdot (t_1,t_2)=(\eta^{2a} t_1, t_2)$.
Note that $H$ naturally embeds in the horizontal torus $\mathbb{G}_m$ acting as $\lambda \cdot (t_1,t_2)=(\lambda t_1,t_2)$.
Furthermore, 
the $\mathbb{G}_m$-action commutes with $\tau_W$, $g_f$ and $g_b$.
Hence, we conclude that $X$ is endowed with a horizontal $\mathbb{G}_m$-action so that $(X,B)$ is
$\mathbb{G}_m$-invariant.
Furthermore, we have that $H\unlhd \mathbb{G}_m H$, by construction.
Hence, $H$ admits a monomorphism into $\mathbb{G}_m$.

\textit{Case 3:} In this case, we assume that $\lfloor B \rfloor $ has two horizontal components $S_0$ and $S_\infty$
which are not $G$-invariant.
In this case, the $G$-orbit of $S_0$ contains $S_\infty$.
Hence, the effective divisor $\lfloor B_Y\rfloor$ has a unique horizontal component $C_Z$ with a ramified morphism to $C_Y$.
Let $Z$ be the normalization of the main component of $Y\times_{C_Y}C_Z$.
Let $(Z,B_Z)$ be the log pull-back of $(Y,B_Y)$ to $Z$. Note that $Y$ is the quotient of $Z$ by an involution $\tau_Z$.
Analogous to the second case, we deduce that $(Z,B_Z)$ is a log Calabi--Yau toric pair.

Let $W$ be the normalization of the main component of $X\times_Y Z$.
Note that $X$ is the quotient of $W$ by an involution $\tau_W$.
Let $(W,B_W)$ be the log pair obtained by pulling-back $(X,B)$ to $W$.
We know that $W\rightarrow Z$ is unramified over $Z\setminus \lfloor B_Z\rfloor$. Hence $(W,B_W)$ is a log Calabi--Yau toric pair.
Note that the involution $\tau_W$ acts on the log pair $(W,B_W)$.
In particular, it restricts to an involution of $\mathbb{G}_m^2$.
Moreover, since $B$ has two horizontal components over $C$, we conclude that
$\tau_W$ must act on $\mathbb{G}_m^2$ as the multiplication by roots of unity.
Hence, we deduce that $(X,B)$ is a toric pair as well.
Thus, $(X,B)$ is a $G$-equivariant log Calabi--Yau toric pair.
\end{proof}

\section{Proof of the main theorems}

In this section, we prove the two main theorems of this article.
We start by proving the case in which the finite automorphism group has large two generation.
Theorem~\ref{introthm-2} is a particular case of the following theorem for $G$-invariant log pairs.

\begin{theorem}\label{thm:rank-2-automorphism}
Let $\Lambda\subset \qq$ be a set satisfying the descending chain condition with rational accumulation points.
There exists a positive integer $N:=N(\Lambda)$, only depending on $\Lambda$,  satisfying the following.
Let $X$ be a Fano type surface and $\Delta$ a boundary on $X$ such that the following conditions hold:
\begin{enumerate}
    \item $G\leqslant {\rm Aut}(X)$ is a finite subgroup with $g_2(G)\geq N$, 
    \item $(X,\Delta)$ is log canonical and $G$-invariant, 
    \item the coefficients of $\Delta$ belong to $\Lambda$, and 
    \item $-(K_X+\Delta)$ is $\qq$-complemented.
\end{enumerate}
Then, there exists:
\begin{enumerate}
    \item A normal abelian subgroup $A\leqslant G$ of index at most $N$,
    \item a boundary $B\geq\Delta$ on $X$, and 
    \item a $A$-equivariant birational morphism $X\dashrightarrow X'$,
\end{enumerate}
satisfying the following conditions:
\begin{enumerate}
    \item The pair $(X,B)$ is log canonical, 
    $G$-invariant, and $(K_X+B)\sim 0$,
    \item the push-forward of $K_X+B$ to $X'$ is a log pair $K_{X'}+B'$, 
    \item $(X',B')$ is a log Calabi--Yau toric pair, and
    \item there are group monomorphisms
    $A<\mathbb{G}^2_m\leqslant {\rm Aut}(X',B')$.
\end{enumerate}
In particular, $B'$ is the reduced toric boundary of $X'$.
\end{theorem}

\begin{proof}
We prove the statement of the theorem in several steps.
We either reduce to the case of del Pezzo surfaces or induce a $G$-equivariant Mori fiber space.
In the latter case, we will have a large $G_f$ and $G_b$ actions. Hence, we can use the techniques developed in subection~\ref{subsec:large-fiber} and subsection~\ref{subsec:large-base} to deduce that in some crepant model, the effective divisor $\lfloor B\rfloor$ has several prime components. Then, we can deduce using the complexity. 

\textit{Step 1:} In this step, we construct an abelian normal subgroup of $G$ of bounded index.
We can apply the Jordan property for finite birational automorphism groups of Fano type varieties~\cite{PS16}.
There exists a normal abelian subgroup $G_0\leqslant G$ whose index is bounded by $J$.
Here, $J$ is a constant that only depends on the dimension of $X$.
By Lemma~\ref{lem:k-gen-subgroup}, we have that $g_2(G_0)\geq g_2(G)/J$.

\textit{Step 2:} In this step, we produce a $G$-equivariant complement and run a $G_0$-equivariant minimal model program.
By Theorem~\ref{thm:g-inv-compl}, we know that there exists a $G$-invariant log canonical $M$-complement
$(X,B)$ for $(X,\Delta)$.
Note that the positive integer $M$ only depends on the set $\Lambda$.
By definition, we have that $(X,B)$ is $G$-invariant, log canonical, and $M(K_X+B)\sim 0$.
In particular, $(X,B)$ is also a $G_0$-invariant log canonical $M$-complement of $(X,\Delta)$.
Let $Y$ be a $G_0$-equivariant log resolution of $(X,\Delta)$.
Let $K_Y+B_Y$ the log pull-back of $K_X+B$ to $Y$.
We may assume that $Y\rightarrow X$ does not extract divisors with log discrepancy larger than one with respect to the pair $(X,\Delta)$. 
Hence, $Y$ remains a Fano type variety.
Observe that the log pair $(Y,B_Y)$ is $G_0$-invariant as well.
We run a $G_0$-equivariant minimal model for $K_Y$.
All the steps of this minimal model program are $G_0$-equivariant crepant contractions for $(Y,B_Y)$.
Since $Y$ is rationally connected, this minimal model program terminates with a $G_0$-equivariant Mori fiber space
$Y'\rightarrow W$.
We denote by $(Y',B_{Y'})$ the push-forward of the log pair $(Y,B_Y)$ to $Y'$.

\textit{Step 3:} In this step, we prove the statement in the case that the $G_0$-Mori fiber space maps to a point.
Assume that this minimal model program terminates with a $G_0$-equivariant Mori fiber space to a point
$Y'\rightarrow W={\rm Spec}(k)$.
Then, we have that $\rho^{G_0}(Y')=1$ and $Y'$ is a smooth Fano type surface.
We conclude that $-K_{Y'}$ is an ample divisor.
Thus, $Y'$ is a $G_0$-equivariant del Pezzo surface.
By Proposition~\ref{prop:g-inv-del-pezzo}, there exists a constant $N_0:=N_0(M)$, so that 
$(Y',B_{Y'})$ is a log Calabi--Yau toric surface, provided that $g_2(G_0)\geq N_0$.
In particular, we have that $K_{Y'}+B_{Y'}\sim 0$.
Thus, we conclude that $K_Y+B_Y\sim 0$ and $K_X+B\sim 0$ holds as well.
Moreover, by Proposition~\ref{prop:g-inv-del-pezzo}, 
there is a characteristic subgroup $A\leqslant G_0$ with a monomorphism 
$A<\mathbb{G}_m^2 \leqslant {\rm Aut}(Y',B_{Y'})$
and the index of $A$ in $G_0$ bounded by $N_0$.
We conclude that $A$ is a normal subgroup of $G$ of index bounded by $N_0J$.
We set $X':=Y'$.
Provided that $N\geq N_0J(2)$, the statement of the theorem holds for the
$A$-equivariant birational map $X\dashrightarrow X'$.

\textit{Step 4:} In this step, we assume that the $G_0$-Mori fiber space maps to a curve.
From now on, we may assume that this minimal model program terminates with a $G_0$-equivariant Mori fiber space to a curve $Y'\rightarrow C$.
We prove that the log pairs induced on the base and the general fiber are log Calabi--Yau toric curves.
Since $C$ is a Fano type curve, we conclude that $C\simeq \pp^1$.
We denote by $F$ the general fiber, which is isomorphic to $\pp^1$.
We have an exact sequence
\[
1\rightarrow G_f \rightarrow G_0 \rightarrow G_b \rightarrow 1,
\]
where $G_f$ is the subgroup of $G$ acting on a general fiber $F$
and $G_b$ is the quotient group of $G$ acting on the base $C$.
We may assume that $G_f$ (resp. $G_b)$ is a cyclic group acting as the multiplication of a root unity on $F\simeq \pp^1$ (resp. $C\simeq \pp^1$).
In particular, $G_0$ is a finite abelian group of rank two.
By Lemma~\ref{lem:cyclic-2-gen}, we conclude that
$\min\{|G_f|,|G_b|\}\geq g_2(G_0)$.
By Lemma~\ref{lem:ab-action-p2}, we conclude that for
$g_2(G_0) \geq 4M$,
the pairs $(F,B_F)$ and $(C,B_C)$ are isomorphic 
to $(\pp^1, \{0\}+\{\infty\})$.

\textit{Step 5:} In this step, we reduce to the case in which the second $G_0$-equivariant extremal contraction of $Y'$ is a fibration.
Assume that the second $G_0$-equivariant extremal contraction of $Y'$ is not a fibration.
Passing to a $G_0$-equivariant crepant birational model of $(Y',B_{Y'})$, we may assume that the fibers over zero and
infinite are $G_0$-prime and contained in $\lfloor B_{Y'}\rfloor$.
By Proposition~\ref{prop:large-base}, we know that $(Y',B_{Y'})$ is a log Calabi--Yau toric pair provided $g_2(G_0)\geq N_1:=N_1(M)$.
In particular, we get that
$K_{Y'}+B_{Y'}\sim 0$ so $K_X+B\sim 0$.
Note that $(Y',B_{Y'})$ is a log Calabi--Yau toric pair and $G_0$ fixes $B_{Y'}$.
Then, we have a monomorphism
$G_0<{\rm Aut}(\mathbb{G}_m^2) \simeq \mathbb{G}_m^2 \rtimes {\rm Gl}_2(\mathbb{Z})$.
By Lemma~\ref{lem:subgrp-glnz}, there exists a characteristic subgroup $A\leqslant G_0$ of bounded index by $l:=l(J)$ so that $A< \mathbb{G}^2_m\leqslant {\rm Aut}(Y',B_{Y'})$.
We conclude that $A\leqslant G$ is a normal subgroup of index bounded by $lJ$
for which $A\leqslant \mathbb{G}^2_m$.
We set $X':=Y'$.
We conclude that the statement of the theorem holds for the 
$A$-equivariant birational map $X\dashrightarrow X'$.

\textit{Step 5:} In this step, we prove the statement in the case that the second $G_0$-equivariant extremal contraction is a fibration.
In this case, we have two $G_0$-equivariant fibrations $\pi_1 \colon Y'\rightarrow C_1$ and $\pi_2 \colon Y'\rightarrow C_2$.
Let $G_{f_i}$ and $G_{b_i}$ be the groups acting on a general fiber and the base of each contraction.
By assumption, we have that 
${\rm min}\{ |G_{f_i}|,|G_{b_i}|\}\geq N$ for each $i$.
Proceeding as in the third step for both contraction, we conclude that the log pairs induced on $C_1$ and $C_2$ are isomorphic to
$(\pp^1,\{0\}+\{\infty\})$.
Applying connectedness of the log canonical centers to $\pi_1$ and $\pi_2$, we conclude that
$B_{Y'}$ has four $G_0$-prime components of coefficient one.
By Theorem~\ref{thm:g-inv-compl}, we conclude that $(Y',B_{Y'})$ is a $G_0$-equivariant toric log Calabi--Yau pair.
In particular, we get that
$K_{Y'}+B_{Y'}\sim 0$ so $K_X+B\sim0$.
Note that $(Y',B_{Y'})$ is a log Calabi--Yau toric pair and $G_0$ fixes $B_{Y'}$.
Then, we have a monomorphism
$G_0<{\rm Aut}(\mathbb{G}_m^2) \simeq \mathbb{G}_m^2 \rtimes {\rm Gl}_n(\mathbb{Z})$.
By Lemma~\ref{lem:subgrp-glnz}, there exists a characteristic subgroup $A\leqslant G_0$ of bounded index by $l:=l(J)$ so that $A< \mathbb{G}^2_m\leqslant {\rm Aut}(Y',B_{Y'})$.
In particular, $A\leqslant G$ is a normal subgroup of index bounded by $lJ$.
We set $X':=Y'$.
We conclude that the statement of the theorem holds for the 
$A$-equivariant birational map $X\dashrightarrow X'$.
\end{proof}

Now, we turn to prove a characterization of Fano type surfaces with large cyclic automorphism groups.
Theorem~\ref{introthm-1} is a particular case of the following theorem for $G$-invariant log pairs.

\begin{theorem}\label{thm:cyclic-automorphism}
Let $\Lambda\subset \qq$ be a set satisfying the descending chain condition with rational accumulation points.
there exists a positive integer $N:=N(\Lambda)$, only depending on $\Lambda$,  satisfying the following.
Let $X$ be a Fano type surface and $\Delta$ a boundary on $X$ such that the following conditions hold:
\begin{enumerate}
    \item $G:=\zz_m\leqslant {\rm Aut}(X)$ with $m \geq N$, 
    \item $(X,\Delta)$ is log canonical and $G$-invariant,
    \item the coefficients of $\Delta$ belong to $\Lambda$, and
    \item $-(K_X+\Delta)$ is $\qq$-complemented.
\end{enumerate}
Then, there exists:
\begin{enumerate}
    \item A subgroup $A\leqslant G$ of index at most $N$,
    \item a boundary $B\geq \Delta$ on $X$, and
    \item a $A$-equivariant birational morphism $X\dashrightarrow X'$,
\end{enumerate}
satisfying the following conditions:
\begin{enumerate}
    \item The pair $(X,B)$ is log canonical, 
    $G$-invariant, and $N(K_X+B)\sim 0$,
    \item the push-forward of $K_X+B$ to $X'$ is a log pair $K_{X'}+B'$, 
    \item $(X',B')$ admits a $\mathbb{G}_m$-action, and
    \item there are group monomorphisms
    $A<\mathbb{G}_m\leq {\rm Aut}(X',B')$.
\end{enumerate}
\end{theorem}

\begin{proof}
We prove the statement in several steps.
The aim is to reduce the statement to the case in which we have a $G$-equivariant Mori fiber space with a large fiber-wise action.
In such a case, we will use Proposition~\ref{prop:G-locally-toric-morphism} to prove that such $G_f$ is embedded in a fiber-wise $\mathbb{G}_m$-action.

\textit{Step 1:} In this step, we produce a $G$-equivariant complement and run a $G$-equivariant minimal model program.
By Theorem~\ref{thm:g-inv-compl}, we know that there exists a $G$-invariant log canonical $M$-complement
$(X,B)$ for $(X,\Delta)$.
By definition, we have that $(X,B)$ is log canonical, $G$-equivariant, and $M(K_X+B)\sim 0$.
We will denote by $Y$ a $G$-equivariant resolution of singularities of $(X,\Delta)$.
Denote by $K_Y+B_Y$ the log pull-back of the log pair $K_X+B$.
Note that $(Y,B_Y)$ is a log canonical pair which is $G$-invariant.
We run a $G$-equivariant minimal model program for $K_Y$.
This minimal model program terminates with a $G$-equivariant
Mori fiber space $Y'\rightarrow W$.
We denote by $(Y',B_{Y'})$ the push-forward of the log Calabi--Yau pair $(Y,B_Y)$ to $Y'$.
Note that $(Y',B_{Y'})$ is still a log Calabi--Yau pair.

\textit{Step 2:} In this step, we prove the statement in the case that the $G$-Mori fiber space maps to a point.
Assume that this minimal model program terminates with a $G$-equivariant Mori fiber space to a point 
$Y'\rightarrow W={\rm Spec}(k)$.
Then, we have that $\rho^G(Y')=1$ and
$Y'$ is a smooth Fano type surface.
Hence, we have that $-K_{Y'}$ is an ample divisor, 
so $Y'$ is a $G$-equivariant del Pezzo surface.
By Proposition~\ref{prop:cyclic-inv-del-pezzo}, 
there exists a constant $N_0:=N_0(M)$, so that  
$|G|\geq N_0$ implies that
the pair $(Y',B_{Y'})$
is a Calabi--Yau toric surface with a $\mathbb{G}_m$-action.
Moreover, by Proposition~\ref{prop:cyclic-inv-del-pezzo}, there exists a subgroup $A\leqslant G$ of index at most $N_0$
satisfying $A<\mathbb{G}_m^2\leqslant {\rm Aut}(Y',B_{Y'})$.
We set $X':=Y'$.
We conclude that the statement of the theorem holds for the 
$G$-equivariant birational map $X\dashrightarrow X'$.

\textit{Step 3:} In this step, we assume that the $G$-Mori fiber space maps to a curve. We prove that either the base or a general fiber is a log Calabi--Yau toric curve.
From now on, we may assume that this $G$-equivariant minimal model program terminates with a $G$-equivariant Mori fiber space to a curve $C$.
We denote by $Y'\rightarrow C$ this $G$-equivariant Mori fiber space.
Then, given that $C$ is a Fano type curve, we conclude that $C\simeq \pp^1$.
On the other hand, the general fiber $F$ is also isomorphic to $\pp^1$.
We have an exact sequence
\[
1\rightarrow G_f \rightarrow G \rightarrow G_b \rightarrow 1, 
\]
where $G_f$ is the subgroup of $G$ acting on the general fiber $F$
and $G_b$ is the quotient group acting on the base $C\simeq \pp^1$.
Up to an isomorphism, we may assume that $G_f$ (resp. $G_b$) acts as the multiplication by a root of unity on $F$ (resp. $C$).
We denote by $(F,B_F)$ the restriction of our log pair $(Y',B_{Y'})$ to the general fiber.
We denote by $(C,B_C)$ the log pair obtained by the $G$-equivariant canonical bundle formula.
We may assume that $M(K_C+B_C)$ (up to replacing $M$ with a bounded multiple).
Since we are assuming that $|G|\geq N$, 
then one of the above groups must be of order at least $N/2$.
Replacing $N$ with $2N$, we can just assume $\max\{|G_f|,|G_b|\}\geq N$.
If $|G_f|\geq N$
and $N\geq 4M$, 
then $(F,B_F)$ is isomorphic to
$(\pp^1,\{0\}+\{\infty\})$.
On the other hand,
if $|G_b|\geq N$
and $N\geq 4M$,
then $(C,B_C)$ is isomorphic to
$(\pp^1,\{0\}+\{\infty\})$.
We will divide the proof in three cases, the horizontal case,
the vertical case, and the mixed case.
In the horizontal (resp. vertical) case we assume that $|G_f|\geq N$ and $|G_b|\leq N$ (resp. $|G_b|\geq N$ and $|G_f|\leq N$).
Finally, in the mixed case, we assume that both groups are large, i.e., $\min \{ |G_f|,|G_b|\}\geq N$.
It is clear that one of these cases must hold.

\textit{Step 4:}
In this step, we prove the statement in the horizontal case.
This means that $|G_b|\geq N$ and $|G_f|$ is bounded by $N$.
We are assuming that $(C,B_C)\simeq (\pp^1,\{0\}+\{\infty\})$.
By performing a $G$-equivariant crepant birational modification of the pair $(Y',B_{Y'})$ we may assume that the fibers over zero and infinite are contained in $\lfloor B_{Y'}\rfloor$.
By Proposition~\ref{prop:large-base}, we may assume that one of the following conditions are satisfies:
\begin{enumerate}
    \item Both extremal rays of the $G$-equivariant cone of curves define fibrations, 
  \item the pair $(X,B)$ is a log Calabi--Yau toric pair, or
\item the pair $(X,B)$ admits a $\mathbb{G}_m$-action, with $H<\mathbb{G}_m\leqslant {\rm Aut}(X,B)$ and $H\leqslant G$ has index bounded by $N$.
\end{enumerate}
Assume that $(1)$ is satisfied.
In this case, we have two $G$-equivariant fibrations $\pi_1\colon Y'\rightarrow C_1$ and $\pi_2\colon Y'\rightarrow C_2$.
Let $G_{f_i}$ and $G_{b_i}$ be the groups acting on a general fiber and base of each contraction.
If $|G_{f_2}|\geq N$, then we reduce to the vertical case.
We may assume that
${\rm min}\{ |G_{b_1}|,|G_{b_2}|\}\geq N$.
We conclude that the log pairs induced on $C_1$ and $C_2$ are isomorphic to
$(\pp^1,\{0\}+\{\infty\})$.
Applying connectedness of the log canonical centers for $\pi_1$ and $\pi_2$, we conclude that
$B_{Y'}$ has two four $G$-invariant irreducible components of coefficient one.
By Theorem~\ref{thm:g-inv-compl}, we conclude that $(Y',B_{Y'})$ is a $G$-equivariant toric log Calabi--Yau pair.
Furthermore, a subgroup of $G$,
of bounded index, acts as the multiplication by a root of unity of a one-dimensional subtorus of $\mathbb{G}^2_m$.
We set $X':=Y'$.
We conclude that the statement of the theorem holds for the 
$G$-equivariant birational map $X\dashrightarrow X'$. It is clear that the statement holds if $(2)$ or $(3)$ are satisfied.

\textit{Step 5:} In this step, we prove the statement in the vertical case.
We turn to prove the vertical case, i.e., 
the case in which $|A_f|\geq N$
and $(F,B_F)$ is a $G_f$-equivariant pair.
We denote by $\pi \colon Y'\rightarrow C$ the $G$-equivariant Mori fiber space.
As before, we have that $(F,B_F)$ is the log Calabi--Yau toric curve.
By~\cite{Amb05}*{Theorem 3.3}, we conclude that the moduli part of the $G$-equivariant canonical bundle formula is a
$\qq$-trivial birational divisor.
In particular, there exists an open set $U\subset C \simeq \pp^1$ and a Galois cover $V\rightarrow U$, so that 
\[
V\times_U (Y',B_{Y'}) \simeq V\times (F,B_F).
\]
In $V\times (F,B_F)$, we have a $\mathbb{G}_m$-action on the second component.
Since $V\times (F,B_F)\rightarrow (\pi^{-1}(U),B_{\pi^{-1}(U)})$
is an \'etale Galois cover, 
we conclude that $(\pi^{-1}(U),B_{\pi^{-1}(U)})$
has a $\mathbb{G}_m$-action as well.
Hence, we have a $\mathbb{G}_m$-action over $U$ which acts fiber-wise.
We may replace $(Y',B_{Y'})$ with a $G$-equivariant dlt modification $(Y'',B_{Y''})$.
Denote by $E$ the vertical components of $\lfloor B_{Y''}\rfloor$.
Note that $E$ is $G$-invariant.
Hence, we may run a $G$-equivariant
$(K_{Y''}+B_{Y''}-\epsilon E)$-minimal model program over $C$ which terminates with a model $(Z,B_Z-\epsilon E_Z)$ over $C$.
Here, as usual, $E_Z$ is just the push-forward of $E$ to $Z$.
Note that $\mathbb{G}_m$ is still acting fiber-wise over an open set of $C$.
On the other hand, by Theorem~\ref{thm:G-locally-toric-surface} and Proposition~\ref{prop:G-locally-toric-morphism},
we conclude that the fibration $(Z,B_Z)\rightarrow C$ is everywhere $G$-equivariantly locally toric over the base,
i.e., for each point $c\in C$ there is a neighborhood of $c\in C$ over which the morphism is toric.
Hence, we conclude that the $\mathbb{G}_m$-action on the general fibers extends to all fibers.
Thus, the fibration $(Z,B_Z)\rightarrow C$ is just the $\mathbb{G}_m$-quotient of a $\mathbb{G}_m$-action on the pair $(Z,B_Z)$.

Finally, observe that $A_f$ embeds in the $\mathbb{G}_m$ of the general fiber. 
Thus, it suffices to define $X':=Z$, and consider
the $G$-equivariant morphism $X\dashrightarrow X'$ which satisfies all the conditions on the statement of the theorem.

\textit{Step 6:} In this step, we prove the statement in the mixed case.
The mixed case follows from the proof of Theorem~\ref{thm:rank-2-automorphism}. In this case, using the same argument as above, we conclude that the model $(Z,B_Z)$ is indeed a toric pair.
Thus, it suffices to define $X':=Z$ and consider
the $G$-equivariant morphism $X\dashrightarrow X'$ which satisfies all the conditions on the statement of the theorem.
\end{proof}

\section{Fano type surfaces with large fundamental group of the smooth locus}
\label{sec:large-fundamental}

In this section, we turn to prove the first application of the main theorems.
We give a characterization of Fano type surfaces with large fundamental group of the log smooth locus.
First, we recall the definition of the fundamental group of the log smooth locus of a Fano type variety.

\begin{definition}{\em 
Let $(X,\Delta)$ be a simple normal crossing pair with standard coefficients.
Let $\mathcal{X}$ be the smooth orbi-fold which realizes $(X,\Delta)$ as the coarse moduli space of $(\mathcal{X},0)$.
We have a surjection of fundamental groups
\[
\pi_1(X\setminus \Delta)\rightarrow \pi_1(\mathcal{X}).
\]
The kernel of this homomorphism is
generated by elements of the form
$\gamma^n$, where $\gamma$ is the loop around a prime component of $D$ with coefficient $1-\frac{1}{n}$.
We may denote the group
$\pi_1(\mathcal{X})$ by $\pi_1(X,\Delta)$
and call it the {\em fundamental group of the pair }$(X,\Delta)$.
Note that the elements of $\pi_1(X,\Delta)$ correspond to finite covers of $X$ on which the log pull-back of $K_X+\Delta$ remains a log pair.

Given a log pair $(X,\Delta)$ with standard coefficients, we denote by $\pi_1(X,\Delta)$ the fundamental group
$\pi_1(X^0,\Delta^0)$.
Here,
$X^0$ is a big open set of $X$ on which the pair $(X,\Delta)$ has simple normal crossing singularities and $\Delta^0$ is the restriction of $\Delta$ to $X^0$.
It is clear that the above definition is independent of $X^0$.

Given an effective divisor $\Delta$ on a variety $X$, we denote by $\Delta_s$ the {\em lower standard approximation of $\Delta$}. This is the largest effective divisor $\Delta_s$ with standard coefficients such that
$\Delta_s\leq \Delta$.
We define $\pi_1(X,\Delta)$ to be $\pi_1(X,\Delta_s)$.
Note that $\pi_1(X,\Delta)$ corresponds to finite covers of $X$ on which the log pull-back of $K_X+\Delta$ remains a log pair.
}
\end{definition}

The following theorem is folklore after the finiteness on the fundamental group of the smooth locus of a weak Fano surface~\cites{FKL93,GZ94,GZ95,KM99}
and the BAB conjecture in dimension two~\cite{Ale94}.

\begin{theorem}\label{thm:finiteness-fg}
Let $X$ be a Fano type surface and $B$ a boundary on $X$ such that $(X,\Delta)$ is klt and $-(K_X+\Delta)$ is $\qq$-complemented.
Then, $\pi_1(X,\Delta)$ is a finite group.
\end{theorem}

\begin{proof}
Since $X$ is a Fano type surface, then it is a Mori dream space.
Furthermore $X$ has $\qq$-factorial singularities.
Replacing $\Delta$ with its lower standard approximation,
we may assume that $\Delta$ has standard coefficients.
We can run a minimal model program for $-(K_X+\Delta)$.
Let $X\rightarrow X_1\rightarrow\dots\rightarrow X_k$ be the steps of this minimal model program.
Note that for each $i$, we have a surjective homomorphism
$\pi_1(X_{i+1},\Delta_{i+1})\rightarrow
\pi_1(X_i,\Delta_i)$.
Hence, it suffices to prove that
$\pi_1(X_k,\Delta_k)$ is finite, where
$-(K_{X_k}+\Delta_k)$ is nef.
Analogosuly, we have a surjective homomorphism
$\pi_1(Y,\Delta_Y)\rightarrow \pi_1(X_k,\Delta_k)$ where $Y$ is the ample model
of $-(K_{X_k}+\Delta_k)$.
The finiteness of $\pi_1(Y,\Delta_Y)$ follows from~\cite{FKL93}.
\end{proof}

Theorem~\ref{introthm:fundamental} is a particular case of the following theorem for log pairs.

\begin{theorem}\label{thm:fundamental-theorem-1}
Let $\Lambda\subset \qq$ be a set satisfying the descending chain condition with rational accumulation points.
There exists a positive integer $N:=N(\Lambda)$,
only depending on $\Lambda$, satisfying the following.
Let $X$ be a Fano type surface.
Let $\Delta$ be a boundary on $X$
with coefficients in $\Lambda$, such that $(X,\Delta)$ is klt, and $-(K_X+\Delta)$ is $\qq$-complemented.
If $G_0\leqslant \pi_1(X,\Delta)$ satisfies 
$g_2(G_0)\geq N$,
then $X$ is log crepant toric quotient.
Furthermore, there exists a finite cover of $(X,\Delta)$ of degree at most $N$ which is log crepant toric.
\end{theorem}

\begin{proof}
By Theorem~\ref{thm:finiteness-fg}, we know that the fundamental group
$G:=\pi_1(X,\Delta)$ is finite.
By Lemma~\ref{lem:g2-supgroup}, we know that $G$ satisfies $g_2(G)\geq N$.
Let $X^0\subset X$ be a big open subset so that
$\pi_1(X,\Delta)=\pi_1(X^0,\Delta_s^0)$ where
$\Delta_s$ is the lower standard approximation of $\Delta^0$.
Let $(Y^0,\Delta^0)$ be the universal cover of $\pi_1(X^0,\Delta_s^0)$.
Define $Y$ to be normal closure of $X$ on the field of fractions of $Y^0$.
We have an automorphism action of $G$ on $Y$.
Note that the action of $G$ on $Y$ is just the regularization of the birational action of $G$ on $Y^0$ as in~\cite{Che04}.
By construction, the pull-back of $(X,\Delta)$ to $Y$ is a log pair.
Let $(X,B)$ be a $M$-complement for $(X,\Delta)$.
Then, the pull-back $(Y,B_Y)$ of $(X,B)$ to $Y$ is a $G$-invariant $M$-complement for the log pair $(Y,\Delta_Y)$.
By Theorem~\ref{thm:rank-2-automorphism}, we deduce that $(Y,B_Y)$ is a crepant toric pair.
Hence, $(X,B)$ is a crepant toric quotient pair.
So $X$ is a log crepant toric quotient variety.

Let $Y\dashrightarrow Y'$ be a $G$-equivariant birational morphism, crepant for $(Y,B_Y)$, such that $(Y',B_{Y'})$ is a $G$-invariant log Calabi--Yau toric pair.
By Theorem~\ref{thm:rank-2-automorphism}, we know that there exists a normal abelian subgroup $A\leqslant G$ of index at most $N$ so that 
$A< \mathbb{G}_m^2$.
Let $(Z,B_Z)$ be the quotient of $(Y,B_Y)$ by $A$.
Then, $(Z,B_Z)$ is log crepant to the quotient $(Z',B_{Z'})$ of $(Y',B_{Y'})$ by $A$.
Since $A$ is a subgroup of the torus of $Y'$, the log pair
$(Z',B_{Z'})$ is still a log Calabi--Yau toric pair.
We conclude that $(Z,B_Z)$ is a cover of $(X,B)$ of degree at most $N$ which is log crepant toric.
\end{proof}

\begin{theorem}
Let $\Lambda\subset \qq$ be a set satisfying the descending chain condition with rational accumulation points.
There exists a positive integer $N:=N(\Lambda)$,
only depending on $\Lambda$, satisfying the following.
Let $X$ be a Fano type surface.
Let $\Delta$ be a boundary on $X$
with coefficients in $\Lambda$, such that $(X,\Delta)$ is klt, and $-(K_X+\Delta)$ is $\qq$-complemented.
If $G_0 \leqslant \pi_1(X,\Delta)$ satisfies $|G_0|\geq N$,
then $(X,\Delta)$ admits a finite cover of degree at most $N$ which has a log crepant torus action.
\end{theorem}

\begin{proof}
By Theorem~\ref{thm:finiteness-fg}, we know that the fundamental group
$G:=\pi_1(X,\Delta)$ is finite.
Let $X^0\subset X$ be a big open subset so that
$\pi_1(X,\Delta)=\pi_1(X^0,\Delta_s^0)$ where
$\Delta^0_s$ is the lower standard approximation of $\Delta^0$.
Let $(Y^0,\Delta^0)$ be the universal cover of $(X^0,\Delta_s^0)$.
Define $Y$ to be normal closure of $X$ on the field of fractions of $Y^0$.
We have an automorphism action of $G$ on $Y$.
Note that the action of $G$ on $Y$ is just the regularization of the birational action of $G$ on $Y^0$ as in~\cite{Che04}.
By construction, the pull-back of $(X,\Delta)$ to $Y$ is a log pair.
Let $(X,B)$ be a $M$-complement for $(X,\Delta)$.
Then, the pull-back $(Y,B_Y)$ of $(X,B)$ to $Y$ is a $G$-invariant $M$-complement for the log pair $(Y,\Delta_Y)$.
If $g_2(G)\geq N$, then we can apply 
Theorem~\ref{thm:fundamental-theorem-1}.
Otherwise, $G$ contains a cyclic subgroup of bounded index.
By Theorem~\ref{thm:cyclic-automorphism}, we deduce that $(Y,B_Y)$ admits a crepant torus action.
Let $(Y',B_{Y'})$ be the crepant model with a torus action.
There exists a subgroup $A\leqslant G$ of inxed at most $N$ so that
$A< \mathbb{G}_m \leq {\rm Aut}(Y',B_{Y'})$.
We conclude that the quotient of $(Y,B_Y)$ by $A$ is a cover of $(X,B)$ of degree at most $N$ which admits a log crepant torus action.
\end{proof}

\section{Degenerations of klt 3-fold singularities}
\label{sec:larg-class-group}

In this section, we study degenerations of klt $3$-fold singularities provided that their local fundamental group is large enough, i.e.,
it has high rank and order of the generators.

In what follows, we will work with the algebraic local fundamental group instead, since the finiteness of local fundamental groups of klt $3$-fold singularities is not known. See, e.g.~\cite{TX17}, for the terminal case.
The algebraic local fundamental group of a klt singularity finite~\cites{Xu14,Sti17}.

\begin{theorem}\label{thm:large-fundamental-group}
Let $\Lambda\subset \qq$ be a set satisfying the descending chain condition with rational accumulation points.
There exists a positive integer $N:=N(\Lambda)$, only depending on $\Lambda$, satisfying the following.
Let $x\in (X,\Delta)$ be a klt $3$-fold singularity so that the coefficients of $\Delta$ belong to $\Lambda$.
If $G\leqslant \pi^{\rm alg}_1((X,\Delta);x)$ satisfies 
$g_3(G)\geq N$, then $x\in (X,\Delta)$ degenerates to a lctq singularity.
\end{theorem}

\begin{proof}
Let $\pi\colon Y\rightarrow X$ be a plt blow-up of $(X,\Delta)$ at $x$.
Let $(X,B)$ be a $M$-complement of $(X,\Delta)$ so that the exceptional divisor of $\pi$ is a log canonical center of $(X,B)$.
Consider the universal cover of $\pi_1^{\rm alg}((X,\Delta);x)$ denoted by $\phi\colon X'\rightarrow X$.
Let $\pi'\colon Y'\rightarrow X'$ be the normalization of the fiber product of the above morphisms.
Note that $\pi'$ is a plt blow-up of $x'\in (X',\Delta')$, where $(X',\Delta')$ is the log pair obtained by pulling-back $(X,\Delta)$ to $X'$.
Furthermore, the log pull-back of $(X,B)$ to $X'$, denote by $(X',B')$, is a $M$-complement.
Note that $x'\in X'$, the pre-image of $x\in X$, is a fixed point for the automorphism group action of $G$ on $X'$.
Hence, $G$ acts as an automorphism group of $Y'$
and it fixes the exceptional divisor $E'$.
By construction, $E'$ is a Fano type surface.
Furthermore, pulling back $(X',B')$ to $Y'$, restricting to $E'$, and applying adjunction,
we obtain a $G_{E'}$-invariant $M$-complement $(E',B_{E'})$.
Here, $G_{E'}$ is the quotient group of $G$ acting as an automorphism group of $E'$.
By Lemma~\ref{lem:g_k-under-quot}, we have that $g_2(G_{E'})\geq N$.
By Theorem~\ref{thm:rank-2-automorphism}, we conclude that $(E',B_{E'})$ is log crepant toric.
We conclude that the quotient $(E,B_E)$ of $(E',B_{E'})$ by $G$ is a log crepant toric quotient.
By~\cite{LX16}*{\S 2.5}, we conclude that $x\in (X,\Delta)$ degenerate to the cone over a log crepant toric quotient.
Thus, $x\in (X,\Delta)$ degenerates to a lctq singularity.
\end{proof}

\begin{theorem}
Let $\Lambda\subset \qq$ be a set satisfying the descending chain condition with rational accumulation points.
There exists a positive integer $N:=N(\Lambda)$, only depending on $\Lambda$, satisfying the following.
Let $x\in (X,\Delta)$ be a klt $3$-fold singularity so that the coefficients of $\Delta$ belong to $\Lambda$.
If $G\leqslant \pi^{\rm alg}_1((X,\Delta);x)$ satisfies 
$g_2(G)\geq N$, then
$x\in (X,\Delta)$ has a finite cover of degree at most $N$ which degenerates to a singularity with a log crepant $\mathbb{G}_m^2$-action.
\end{theorem}

\begin{proof}
Let $\pi\colon Y\rightarrow X$ be a plt blow-up of $(X,\Delta)$ at $x$.
Let $(X,B)$ be a $M$-complement of $(X,\Delta)$ so that the exceptional divisor of $\pi$ is a log canonical center of $(X,B)$.
Consider the universal cover of $\pi_1^{\rm alg}((X,\Delta);x)$ denoted by $\phi\colon X'\rightarrow X$.
Let $\pi'\colon Y'\rightarrow X'$ be the normalization of the fiber product of the above morphisms.
Note that $\pi'$ is a plt blow-up of $x'\in (X',\Delta')$, where $(X',\Delta')$ is the log pair obtained by pulling-back $(X,\Delta)$ to $X'$.
Furthermore, the log pull-back of $(X,B)$ to $X'$, denote by $(X',B')$, is a $M$-complement.
Note that $x'\in X'$, the pre-image of $x\in X$, is a fixed point for the automorphism group action of $G$ on $X'$.
$G$ acts as an automorphism group of $Y'$
and it fixes the exceptional divisor $E'$.
By construction, $E'$ is a Fano type surface.
Furthermore, pulling back $(X',B')$ to $Y'$, restricting to $E'$, and applying adjunction,
we obtain a $G_{E'}$-invariant $M$-complement $(E',B_{E'})$.
Here, $G_{E'}$ is the quotient group of $G$ acting as an automorphism group of $E'$.
Observe that $g_1(G_{E'})\geq 1$.
By Theorem~\ref{thm:cyclic-automorphism}, we conclude that $(E',B_{E'})$ admits a crepant torus action.
Furthermore, there is a subgroup $A\leqslant G_{E'}$ of bounded index so that the quotient of $(E',B_{E'})$ by $A$ still admits a torus action.
Let $A_G\leqslant G$ be a normal subgroup of $G$ surjecting onto $A$.
By~\cite{LX16}*{\S 2.5}, the cover of $x\in (X,\Delta)$ corresponding to $A_G$ degenerates to a cone over a pair with crepant $\mathbb{G}_m$-action.
Hence, $x\in (X,\Delta)$ degenerates to a singularity with a crepant $\mathbb{G}_m^2$-action.
\end{proof}

\begin{proof}[Proof of Theorem~\ref{introthm:large-class}]
Note that we have an injective homomorphism
${\rm Cl}(X;x)\rightarrow \pi_1^{\rm alg}((X,\Delta);x)$ induced by the index one cover of Weil $\qq$-Cartier divisors on $X$.
Then, the statement follows from Theorem~\ref{thm:large-fundamental-group}
\end{proof}

\end{document}